\documentclass[a4paper,11pt]{amsart} 
\usepackage{MiscStyle}
\usepackage{GeomA4Style}
\usepackage{NumberArtStyle2}
\usepackage{ShortcutStyle}
\begin{document}
\author{Guillermo Pe\~nafort Sanchis}

\title {Reflection Maps}








\begin{abstract} 
Given a reflection group $G$ acting on a complex vector space $V$, a reflection map is the composition of an embedding $X \hookrightarrow V$ with the orbit map $V\to\C^p$ that maps a $G$-orbit to a point. Reflection maps can be very singular, but we give tools to study them easily. We find obstructions to $\cA$-stability of reflection maps and produce, in the unobstructed cases, infinite families of $\cA$-finite map-germs of any corank. We also relate them to conjectures of L\^e, Mond and Ruas.
\end{abstract}

\maketitle


								\section{Introduction}

%
%
%

Our ability to understand a singular map $f\colon X\to Y$ between manifolds, locally around some point $x\in X$, depends drastically on how far the differential $df_x$ is from having its maximum possible rank. This loss of rank is called the \emph{corank} of $f$ at $x$ (Definition \ref{defCorank}) and finite maps having corank at most one everywhere are sometimes called \emph{curvilinear}. Many subjects related to singular maps, from the multiple point formulas \cite{KleimanMultiplePointFormulasI} and  Thom polynomials \cite{Berczi:2012,Ohmoto:2014}, to the vanishing homology \cite{GoryunovMond,HoustonTop} and  Whitney equisingularity \cite{Gaffney1993Polar-multiplic} of families of map-germs, are far better understood under the corank one hypothesis. The following simple yet surprisingly difficult open question, posted by L\^e (Conjecture \ref{conjLe}), exemplifies well how little we know about non-curvilinear maps:
 \emph{Is there any injective germ $(\C^2,0)\to (\C^3,0)$ of corank two?}
 
The objects employed to study map-germs, such as the multiple point spaces or  the instability locus,  usually have   an intricate algebraic structure, making it hard for us to predict the properties of a map by taking a glance at its coordinate functions. As one may imagine, the number of variables that play an essential role in computations is proportional to the corank of the map and the degree of the equations involved tends to increase with that of the coordinate functions. These technical obstacles have restricted most of the attention to the curvilinear case, to the point that there are few examples of $\cA$-finite map-germs with greater corank (some can be found in \cite{Altintas:2014,AltintasThesis,MararNunoANoteOnFiniteDeterminacyForCorank2, Mond:2016}). They have also left us with no known $\cA$-finite family containing examples with as high multiplicity as desired (here the multiplicity is the minimum of the orders of the coordinate functions), with the only exception, to the best of my knowledge, of the germs of curves. However,  the interest on non-curvilinear singularities has increased substantially over the last years (see, for instance, \cite{Altintas:2014, AltintasThesis, Feher2006On-the-second-o, Feher2012Thom-series-of-, Fernandez-de-Bobadilla2006A-reformulation, Marangell2010The-General-Qua, Marar2012Double-point-cu, Mond:2016, Ohmoto2014Bifurcation-of-, Rimany2002Multiple-Point-}) and it has been shown that, to proof Mond's conjecture (this is one of the main open problems in vanishing homology, see Section \ref{secFinalRemarks}), it suffices to check that it holds on a family with unbounded multiplicity. The aim of this work is to show that we may use reflection groups to produce maps of high corank and multiplicity, the \emph{reflection maps} of the title, which we understand better. 

Reflection maps are produced by taking an embedding $h\colon X\hookrightarrow V$ and `folding' it in a particular way,  prescribed by a reflection group $G$.  The relation between $G$ and $h$ determines the corank of the resulting map and groups with a bigger order may produce maps with higher multiplicity. These maps have the advantage that the action of $G$ breaks algebraically hard conditions into simpler pieces involving just the embedding $h$ and its translates by elements of the group. 

Many  singular maps in existing classifications are reflection maps (see Example \ref{exSimpleC2C3} and Section \ref{secFinalRemarks}) and their study could benefit from this perspective. More importantly, reflection maps have proved fruitful as a source of interesting new examples: They are among the first known counterexamples, as shown by Silva and Ruas \cite{Ruas:2017}, to a conjecture by Ruas \cite{Ruas:1994}, which stated the equivalence between topological triviality and Whitney equisingularity of families of map-germs $\C^2\to\C^3$ (they are in Example \ref{exNTo2NMinus1}, but we won't get into the details of this). Reflection maps also contain $\cA$-finite germs of maps $\C^n\to\C^p$ with arbitrarily high multiplicity, for $p\geq 2n-1$, and the ones from $\C^2$ to $\C^3$ might contribute to solving cases of Mond's conjecture for corank one and two (see Section \ref{secFinalRemarks} for details). As for L\^e's question, we will solve a much more general one for reflection maps, motivating an extension of L\^e's conjecture (see Section \ref{secCorankAndInj}) about the relation between corank, injectivity and the dimensions of source and target of map-germs.

Before getting into technicalities, we shall introduce the most basic reflection maps and explain what it means to `fold' an embedding. The simplest family of reflection maps is the fold map family, introduced by Mond \cite{Mond1985On-the-clasific} in the classification of $\cA$-simple germs $(\C^2,0)\to (\C^3,0)$.

	\begin{definition}\label{defFoldMap}A \emph{fold map} is a germ $f\colon (\C^n,0)\to(\C^p,0)$ of the form
\[x\mapsto (x_1,\dots,x_{n-1},x_{n}{}^2,H(x)),\]
for some $H\colon (\C^n,0)\to (\C^{p-n},0)$.
	\end{definition}
	
A fold map can be regarded as the composition of the graph of $H$, $x\mapsto (x,H(x))$, with the map $\omega\colon (\C^p,0)\to(\C^p,0)$, given by
	\[y\mapsto (y_1,\dots,y_{n-1},y_n{}^2,y_{n+1},\dots,y_p).\]
If $G$ is the group generated by the reflection with respect the hyperplane $\{y_n=0\}$, then $\omega$ maps a $G$-orbit to a point.  Of course, the graph of $H$ is an embedding and composing with $\omega$ is what we mean by folding it by means of $G$. An example of fold map is the cusp $x\mapsto (x^2,x^3)$, obtained by folding the graph of the function $H(x)=x^3$, as in Figure \ref{figCuspAsFold}.

 \begin{figure}
\begin{center}
\includegraphics[scale=1.2]{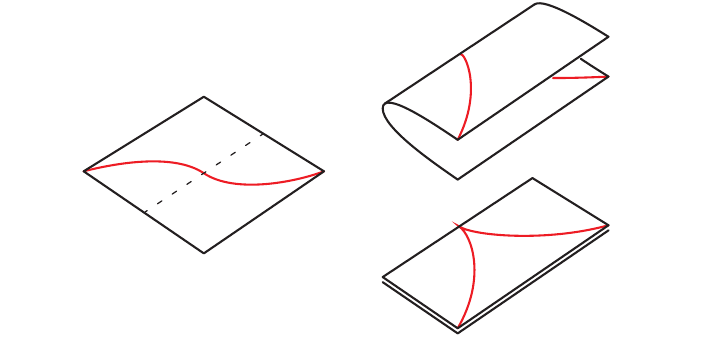}
\end{center}
\caption{The cusp as a fold map.}
\label{figCuspAsFold}
\end{figure}

Fold maps have been studied widely (see \cite{Wilkinson:1991} and  \cite{Houston:1998}) and it follows from work of Whitney \cite{Whitney44TheSingularitiesOfSmooth, Whitney:1943} that all stable singular maps germs $f\colon (\C^n,0)\to (\C^{2n-1},0)$ are fold maps. Mond showed that the $\cA$-classification of fold maps is equivalent to the $\cK^T$-classification of $H$, already carried out by Arnol'd \cite{Arnold:1978}, where the $\cK^T$-equivalence is an adapted contact equivalence preserving the hyperplane $\{y=0\}$. In other words, the classes of $\cA$-equivalence of fold maps are related to the contact between the graph of $H$ and the hyperplane $\{y=0\}$. The next family of reflection maps was introduced by Marar and Nu\~no Ballesteros \cite{MararNunoANoteOnFiniteDeterminacyForCorank2}, in the search of $\cA$-finite germs $(\C^2,0)\to(\C^3,0)$ of corank 2.

	\begin{definition}\label{defDoubleFoldMap}A \emph{double-fold map} is a germ $f\colon (\C^n,0)\to(\C^p,0)$ of the form
\[x\mapsto (x_1,\dots,x_{n-1}{}^2,x_{n}{}^2,H(x)),\]
for some $H\colon (\C^n,0)\to (\C^{p-n},0)$.
	\end{definition}

Double-fold maps are much like fold maps, but they are folded by means of the group generated by the reflections $r_1$ and $r_2$ with respect to the hyperplanes $\{y_{n-1}=0\}$ and $\{y_n=0\}$. In this case, folding means composing with the \emph{Folded Hankerchief} $\omega\colon (\C^p,0)\to(\C^p,0)$, depicted in Figure \ref{figC2xC2Orbitmap} and given by
	\[y\mapsto (y_1,\dots,y_{n-1}{}^2,y_n{}^2,y_{n+1},\dots,y_p).\]
 \begin{figure}
\begin{center}
\includegraphics[scale=1.2]{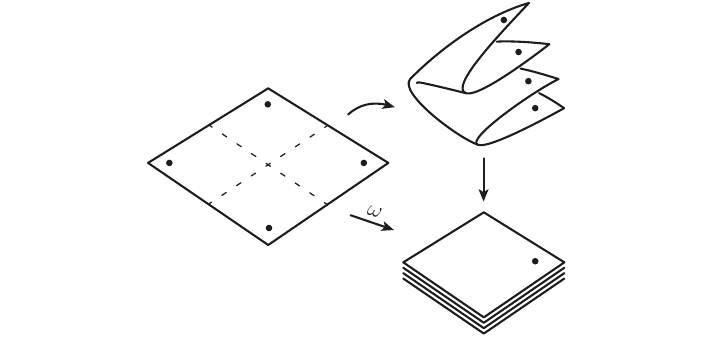}
\end{center}
\caption{The Folded Hankerchief.}
\label{figC2xC2Orbitmap}
\end{figure}
In \cite{Penafort-Sanchis2014THE-GEOMETRY-OF}, I studied double-fold maps $(\C^2,0)\to (\C^3,0)$ in terms of the contact of the graph of $H$ and strata determined by the hyperplanes $\{x=0\}$ and $\{y=0\}$. Many of the ideas in the present paper come from the ones developed there.

A natural way to generalize fold and double-fold maps is the following: Pick any embedding $h\colon X\hookrightarrow V$ and any group $G$ acting on $V$, then fold $h(X)$ so that points lying on the same $G$-orbit are glued. However, to come up with a map between complex manifolds, we need to find a holomorphic map $\omega\colon V\to Y$ which produces this folding. As Hilbert showed \cite{Hilbert:1890}, the invariants of a finite group acting linearly on a $\C$-vector space $V$ can be algebraically generated by a finite number $p\geq \dim V$ of polynomials. By a result of Noether (see Theorem \ref{thmAlphaGeomInv}), our desired folding is acomplished by composing with the map $\omega\colon V\to \C^{p}$ whose coordinate functions are the generating invariants of $G$. The reason why we ask for $G$ to be a reflection group is to take advantage of the rich theory regarding them. The part of the theory which is most relevant to us has to do with a stratification of $V$, induced by the reflecting hyperplanes of $G$, and how it encodes the singularities of $\omega$ and the action of $G$.
Another pleasant feature of reflection groups is that, by Shephard-Todd's Theorem \ref{thmShephardTodd}, these are the ones with a number $p=\dim V$ of generating invariants, so that the dimension of the target of $h$ is preserved after the folding $\omega$.

\subsection*{Outline of the work} Section \ref{secRefGroups} contains, briefly and without proofs, preliminary materials on reflection groups. 
 In Section \ref{secRefMaps} we define reflection maps and show some basic results.
 
  Section \ref{secCorankAndInj} deals with the relation between injectivity and  the corank $k$ locus of reflection maps. We prove that reflection maps satisfy the following extended version of L\^e's conjecture: There is no injective germ of reflection map $f\colon (\C^n,0)\to (\C^p,0)$, with $\corank f>p-n$.
  
   The main result of  Section \ref{secNormalCrossings} is Theorem \ref{thmCharacNormalCrossings}, which characterizes  normal crossings of reflection maps  in terms of two properties determined by the relation between the embedding $h$ and the group $G$. From this we obtain obstructions for reflection maps to have normal crossings and hence to be stable.
 
  The remaining sections are devoted to the study of stability and $\cA$-finiteness of reflection maps by means of their double-point spaces. In Sections \ref{secD2Omega} and \ref{secB2RefMaps} we introduce and describe the  double point spaces $D^2(\omega)$ and $B^2(\omega)$ of the orbit map $\omega$. Both spaces have a primary decomposition indexed by $G\setminus\{1\}$ (Theorems \ref{thmD2OrbitMap} and \ref{thmB2Omega}). We show how $B_2(\omega)$ determines the double point space $B_2(f)$ of a reflection map. It is worth noting that all these spaces can be computed without knowing the expression of the orbit map $\omega$. Computing these double points, without the extra structure provided by $G$, would be a much harder case by case task.
  
  Sections \ref{secObstrAFin} and \ref{secAFiniteRefMaps} contain the main results of this paper, which are the obstructions and criteria for stability and $\cA$-finiteness of reflection maps and the families of $\cA$-finite reflection maps of rank 0. In Section \ref{secObstrAFin} several previous results are combined to obtain the following: 
  \begin{enumerate}
\item There are no stable germs of reflection map of corank $\geq 2$ and all stable essential reflection maps of corank one are fold maps (Theorem \ref{thmNoRMStabCor2}).
\item For $p<2n-1$, there are no $\cA$-finite germs of reflection map $f\colon (\C^n,0)\to (\C^p,0)$ with $\corank f\geq 2$ and all $\cA$-finite essential reflection maps of corank one are fold maps (Theorem \ref{thmNoRMAFinCor2Dims}).
  \end{enumerate}
In Section \ref{secAFiniteRefMaps} we introduce criteria for $\cA$-finiteness in arbitrary corank in the non-obstructed dimensions $p\geq 2n-1$ in terms of the double-point space $B^2(f)$ (Propositions \ref{propCharAFinPgeq2N} and \ref{propAfiniteNTo2N-1}). These criteria, together with the explicit description of the branches $B_g(h)$ and some technical results contained in Section \ref{subsecAuxResults}, allow us to give infinite families of rank 0 $\cA$-finite map-germs for $p=2n$ and $p=2n-1$. For instance, for pairwise coprime integers $m_i$, we show that the following germs are $\cA$-finite:
\begin{align*}
x&\mapsto(x^{m_1},x^{m_2})\\
(x_1,x_2)&\mapsto(x_1{}^{m_1},x_2{}^{m_2},(x_1+x_2)^{m_3})\\
(x_1,x_2)&\mapsto(x_1{}^{m_1},x_2{}^{m_2},(x_1+x_2)^{m_3},(x_1-x_2)^{m_4})\\
(x_1,x_2,x_3)&\mapsto(x_1{}^{m_1},x_2{}^{m_2},x_3{}^{m_3},(x_1+x_2+x_3)^{m_4},(x_1-x_2+2x_3)^{m_5},(x_1+2x_2-x_3)^{m_6}).
\end{align*}
We make some final comments in Section \ref{secFinalRemarks}. We review the presence of reflection maps in well known classifications of map-germs. We motivate the search examples of $\cA$-finite map-germs of corank 2, by relating it to a new approach to Mond's conjecture. Finally, we discuss ideas for future projects related to reflection maps.

For simplicity, this work is written over $\C$, but most results should apply to $\R$ with little or no change.  Some others, such as the characterization of injectivity in Lemma \ref{lemCharInjRefEmb}, require replacing (complex) reflection groups by Coxeter groups.

\begin{notation*}
Throughout the text, $f\colon X\to Y$ stands for a finite holomorphic map between complex manifolds.
\end{notation*}
				
							\section*{Aknowledgements}
							
I thank Ton Marar and Juan Jos\'e Nu\~no Ballesteros, who introduced me to double-fold maps;  David Mond, who suggested to generalize the results using reflection groups; Javier Fern\'andez de Bobadilla, responsible for the general definition of reflection maps; Atoshi Chowdhury and Carolina Vallejo, for their help on Section \ref{subsecAuxResults} and Bruna Or\'efice Okamoto, for many useful conversations.
					
								\section{Preliminaries on reflection groups}\label{secRefGroups}
There is a vast theory regarding reflection groups but we will only explain here what is needed to get our results.  A general reference for the topic is \cite{LehrerTaylorUniRefGroups} and corresponding results over $\R$ can be found in \cite{Humphreys1990Reflection-Grou}.

Given  be a group $G$ acting on a set $X$, we write $\Fix g=\{x\in X\mid gx=x\}.$
We write the orbit of a subset $S$ of $X$ as $GS$ and drop unnecessary brackets, so that $Gx=G\{x\}$.
For every $g\in G$, we write $gS=\{gy\mid y\in S\}$. We fix a positive definite hermitian form $\langle-,-\rangle$ on $V$, given by 
$(\sum a_ie_i,\sum b_je_j)\mapsto \sum a_i\bar{b_i}$, where $e_i$ is the $i$th vector of a fixed basis of $V$. A linear map $g\colon V\to V$ is \emph{unitary} if $\langle y,y'\rangle=\langle gy,gy'\rangle$, for all $y, y'\in V$. 

	\begin{definition}
A \emph{reflection} on $V$ is a linear map $r\colon V\to V$, satisfying
\begin{enumerate}
\item $r$ is unitary, 
\item $r$ has finite order,
\item  $\dim \Fix  r=\dim V-1$. 
\end{enumerate}
The subspace  $H_r=\Fix  r\leq V$ is the \emph{reflecting hyperplane} of $r$.
	\end{definition}
	
	It is well known that every reflection $r$ acts trivially on $H_r$, and by multiplication by a finite-order root of the unity on the subspace $H_r^\bot$.

We write ${\rm U}(V)$ for the group of unitary automorphisms of $V$. When we say that $G$ is a subgroup of ${\rm U}(V)$, it is assumed that $G$ acts by restriction of the action of ${\rm U}(V)$. When two subgroups of ${\rm U}(V)$ are called isomorphic, the isomorphism is assumed compatible with the actions. This is relevant, as some non-isomorphic reflection groups are isomorphic as abstract groups.

	\begin{definition}
 A subgroup $G$ of $U(V)$ is a  \emph{reflection group} if it is finite and can be generated by reflections.
	\end{definition}

						
	\begin{definition}
The \emph{rank} of $G$ is the dimension of the vector subspace $W\leq V$ spanned by the orthogonal spaces of the reflecting hyperplanes of all reflections in $G$, or by the orthogonal spaces of the reflecting hyperplanes of any generating system of reflections. The group $G$ acts on $W$ and fixes $W^\bot$ pointwise.
	\end{definition}



Given two groups $G_1$ and $G_2$ acting on $V_1$ and $V_2$, the product $G=G_1\times G_2$ acts as a reflection group on $V_1\oplus V_2$. Reflections on each $V_i$ become reflections on $V_1\oplus V_2$, by extending their actions trivially. Obviously, $G_1\times G_2$ is generated by the (extensions of) two respective generating sets of reflections of $G_1$ and $G_2$.
Reflection groups that cannot be obtained as product of nontrivial reflection groups are called irreducible, and were classified by Shephard and Todd \cite{Shephard1954Finite-unitary-}. The list consists of the infinite family $G(m,p,n)$, and 34 `sporadic' reflection groups.
	
	\begin{ex}\label{exZm1mp}
The cyclic group $Z_m=\Z/m\Z$ acts on $\C$ (and on $\R$, if $m=2$) by multiplication by powers  of a primitive $m$th root of unity. We label the elements in $Z_m$ by numbers $0\leq a< m$, so that the action of the element $i_a$ is $x\mapsto e^{\frac{2\pi i}{m} a} x$.

We write $Z_{m_1,\dots,m_p}=Z_{m_1}\times\dots \times Z_{m_p}$. Elements in $Z_{m_1,\dots,m_p}$ are labelled $i_a=i_{a_1,\dots,a_p}$, with $0\leq a_i<m_i$ and $a=(a_1,\dots,a_p)$. Each $i_a$ is a reflection if and only if exactly one $a_i$ is not zero. 

	\end{ex}

						\subsection*{The complex of a reflection group}\label{secFacets}

The way a reflection group acts changes over $V$ according to the following natural stratification: The \emph{arrangement of hyperplanes} of $G$ is the union 
\[\sA=\bigcup_{i=1}^k H_i,\]
of the reflecting hyperplanes $H_1,\dots, H_k$ of all reflections in $G$. 	
The reflecting hyperplanes $H_1,\dots, H_k$ induce the following partition of $V$.
For each subset $B\subseteq \{1,\dots,k\}$, set
\[C_B=\{y\in V\mid y\in H_i\Leftrightarrow i\in B\}.\]
 Each nonempty set $C_B$ is called a \emph{facet} and  the set  $\sC$ of all facets is the \emph{complex of $G$}. We drop brackets and write the elements in $B$ in order, so that $C_{1,2}$ stands for $C_{\{2,1\}}$.

Every point $y\in V$ is contained in a facet $C\in \sC$, consisting of the points contained in exactly the same hyperplanes as $y$. Obviously, $\sC$ forms a partition of $V$. The way facets are defined, it is clear that the closure of a facet $C_B$ is the vector subspace $\langle C_B\rangle$ spanned by $C_B$, and equals the intersection of the hyperplanes $H_i,$ with $i\in B$. 

\begin{caution*}
The complex just defined is not the Coxeter complex of a \emph{real} reflection group. The Coxeter complex, explained in  \cite{Humphreys1990Reflection-Grou}, is a finer stratification where positivity questions are taken into account. 
\end{caution*}


The following result, which is just a restatement of Corollary 9.14 in \cite{LehrerTaylorUniRefGroups}, is key to many of the results in this paper:

	\begin{prop}\label{propElementFixesFacet}
For every facet $C\in \sC$, there exists an element $g\in G$, such that $\Fix  g=\langle C\rangle$.
	\end{prop}

	\begin{ex}\label{exRefHypZm1mp}
The reflecting hyperplanes of $Z_{m_1,\dots,m_p}$ are just $H_i=\{y_i=0\}$.
Every choice of hyperplanes $B\subseteq \{1,\dots,p\}$ yields a facet $C_B=\{y_i=0\Leftrightarrow i\in B\},$
 whose closure is $\langle C_B\rangle =\{y_i=0,\text{ for all }i\in B\}.$
 For a facet $C_B$, take any $i_a\in Z_{m_1,\dots,m_p}$, with $a_i\neq 0$ if and only if $i\in B$. It is immediate that $\Fix\ i_a=\langle C_B\rangle$. The complex  of $Z_{m_1,m_2}$ is depicted in Figure \ref{figC2xC2Complex}.
\begin{figure}
\begin{center}
\includegraphics[scale=1.2]{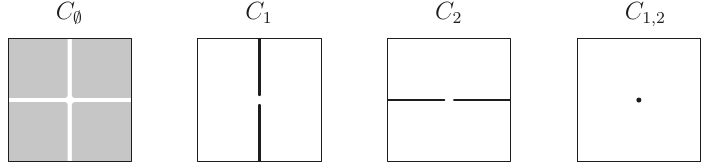}
\end{center}
\caption{Facets of $Z_{m_1,m_2}$.}
\label{figC2xC2Complex}
\end{figure}
	\end{ex}

	\begin{thm}[Steinberg,\cite{Steinberg:1968}]  \label{thmStabilizerReflection}
For any subset $A\subseteq V$, the pointwise stabilizer
\[G_A=\{g\in G\mid ga=a,\text{ for all }a \in A\}\]
is a reflection group, generated by the reflections in $G$ whose reflecting hyperplanes contain $A$.
	\end{thm}

	\begin{cor}\label{corStabilizer}
	\begin{enumerate}
\item \label{corStabilizerItem1} For any $y\in V\setminus \sA$, we have $G_y=1$.
 \item \label{corStabilizerItem2}	If $y\in V$ is a point in a facet $C\in \sC$, then $G_y=G_C$. 
		\end{enumerate}
\begin{proof}The first item is immediate. The second follows by construction of the facets, because the reflecting hyperplanes containing $y$ are exactly the ones containing $C$. 
\end{proof}
	\end{cor}

							\subsection*{The orbit map}\label{secOrbitMap}As we already mentioned in the Introduction, a group $G$ acting on $V$ determines a way to `fold' $V$, glueing a $G$-orbit to a point. Such a folding is prescribed by any system of generators of the algebra $\C[y_1,\dots,y_p]^G$ of $G$-invariant polynomials on $V$. The following result of Shephard and Todd \cite{Shephard1954Finite-unitary-}, with a more elegant proof due to Chevalley \cite{Chevalley:1955}, shows that reflection groups are characterized by having the simplest invariant algebras.
	\begin{thm}\label{thmShephardTodd}
Let $G$ be a finite subgroup of $\rm U(V)$ and $p=\dim V$. Then $G$ is a reflection group if and only if there exist $p$ homogeneous algebraically independent polynomials $\omega_1,\dots,\omega_p$, such that $\C[y_1,\dots,y_p]^G=\C[\omega_1,\dots,\omega_p]$.	
	\end{thm}
	
	Furthermore, the degrees $d_i=\deg(\omega_i)$ of the generating invariants $\omega_i$ of a reflection group $G$ are uniquely determined and satisfy $\prod d_i=\vert G\vert$ and that the number of reflections in $G$ is $\sum{d_i-1}$ \cite{Shephard1954Finite-unitary-}. There are many results other results on the invariants of reflection groups, we will just mention a criterion by Springer \cite{Springer:1974}, which gives a transparent way to ensure that you have a generating system of invariants: Let $G$ be a reflection group acting on $V$ and let $p=\dim V$. If $\omega_1,\dots,\omega_p$ are homogeneous algebraically independent invariant polynomials and $\prod \deg \omega_i=\vert G\vert$, then  $\omega_i$ generate the algebra of $G$-invariant polynomials on $V$.
	\begin{definition}
The \emph{orbit map} of a reflection group $G$ is the map $\omega\colon V\to\C^p$ whose coordinate functions are the generating invariants $\omega_1,\dots,\omega_p$.
	\end{definition}

The map $\omega$ is unique up to invertible polynomial transformations in the target.  Since we work with objects which are invariant under such transformations, the choice of $\omega$ does not matter. This justifies us abusively calling $\omega$ \emph{the} orbit map of $G$.  The name of the orbit map is motivated by the following theorem, due to Noether \cite{Noether:1916}. The result is true in a more general situation, namely when the coordinate functions of $\omega$ generate the invariants of a finite group acting lineraly on $V$.

	\begin{thm}\label{thmAlphaGeomInv}For any $y\in V$, we have
$\omega^{-1}(\omega(y))=Gy.$
	\end{thm}




	\begin{ex}\label{exOrbitMapZm1mp}
The orbit map $\omega\colon V\to \C^p$ of the group $Z_{m_1,\dots,m_p}$ is given by \[(y_1,\dots,y_p)\mapsto(y_1{}^{m_1},\dots,y_p{}^{m_p}).\]
In particular, the above mentioned Folded Hankerchief is the orbit map of the group $Z_{2,2}$. 
	\end{ex}

To control the corank of our reflection maps, we will need the following lemma, which follows directly from the proof in \cite[Theorem 9.13]{LehrerTaylorUniRefGroups} of a result originally due to Steinberg \cite{Steinberg:1960}.	 Since the objects involved are of linear nature, we may simplify the differential notation:
For any $y\in V$, we identify the tangent space $T_yV$ with $V$ itself, and regard  $d\omega_y$  as a linear map $V\to \C^p$. Since the tangent of a facet $C\in \sC$ is the same at every point, we write its orthogonal complement (at any point) as $C^\bot\leq V$.

	 \begin{lem}\label{KerDifOrbitMap}
For any $y\in V$ in a facet $C\in \sC$, we have that $\ker d\omega_y= C^\bot.$ In particular, the set of points $y\in V$ where $d\omega_y$ does not have rank  $p$ is the arrangement of hyperplanes $\sA$.
	 \end{lem}

	\section{Reflection maps}\label{secRefMaps}
Here we define (quasi-) reflection maps, essential reflection maps and reflected graphs. We give examples and show some basic properties.
	\begin{definition}
Given a reflection group $G$ acting on $V$, a \emph{$G$-reflection map} $f\colon X\to\C^p$ is a composite 
	\[f=\omega\circ h,\]
where $h\colon X\hookrightarrow V$ is an embedding and  $\omega\colon V\to \C^p$ is the orbit map of $G$. If we only ask $h\colon X\to V$ to be a finite map, then $f=\omega\circ h$ is a \emph{quasi-reflection map}.
	\end{definition}
	
	\begin{notation*}
In what follows, whenever we say that $f\colon X\to \C^p$ is a reflection map, then $G$, $h$ and $\omega$ stand  for the corresponding reflection group, embedding and orbit map. We also write
\[Y=\im h\subseteq V.\]	\end{notation*}
	
	\begin{definition}\label{defAEquiv}
Two maps $f\colon X\to Y$ and $f'\colon X'\to Y'$ are $\cA$-equivalent if there exist biholomorphisms  $\phi\colon X'\to X$ and $\psi\colon Y'\to Y$, such that $f'=\psi\circ f\circ\phi$. The definition extends to multi-germs by taking representatives.
	\end{definition}

Observe that since $\omega$ is unique up to invertible polynomial transformations in the target, the $\cA$-class of a $G$-reflection map does not depend on the choice of the orbit map $\omega$.

	\begin{ex}
The following are three different ways to obtain a parametrized a cusp 
 $t\mapsto (t^2,t^3)$
as a reflection map (Figure \ref{figCuspThreeWays}):
\begin{enumerate}
\item  $h(t)=(t,t^3)$ and $\omega(X,Y)=(X^2,Y)$, with reflection group $Z_2$,
\item  $h(t)=(t^2,t)$ and $\omega(X,Y)=(X,Y^3)$, with reflection group $Z_3$,
\item  $h(t)=(t,t)$ and $\omega(X,Y)=(X^2,Y^3)$, with reflection group $Z_{2,3}$.
\end{enumerate}
 \begin{figure}
\begin{center}
\includegraphics[scale=1.2]{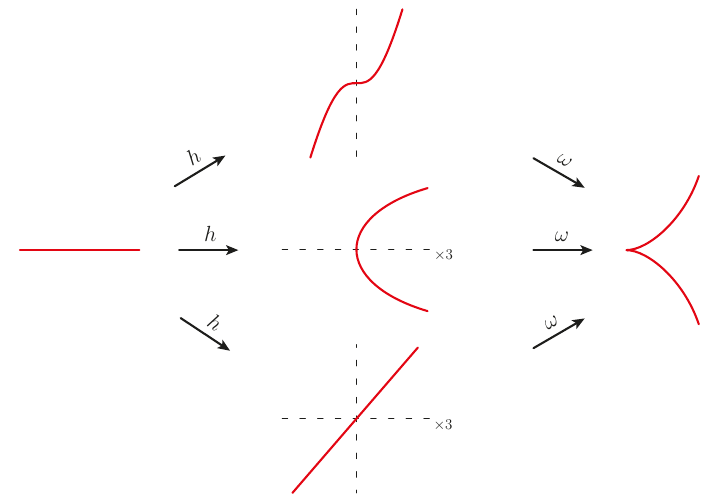}
\end{center}
\caption{Three ways to obtain a cusp as a reflection map.}
\label{figCuspThreeWays}
\end{figure}

	\end{ex}

	\begin{ex}
Clearly, fold maps (Definition \ref{defFoldMap}) are $Z_2$-reflection maps and double-fold maps (Definition \ref{defDoubleFoldMap}) are $Z_{2,2}$-reflection maps. As we saw in the previous example, the families of reflection maps are not mutually exclusive. For instance, the Cross-Cap can be parametrized as a fold map $(x,y)\mapsto(x,y^2,xy),$
or as a double-fold map $(x^2,y^2,x+y).$
A detailed study of double-fold maps can be found in \cite{Penafort-Sanchis2014THE-GEOMETRY-OF}. From this work follows that a generic  choice of a function $h\in \cO_2$ (in the sense that some transversality conditions are met) of order two produces an $\cA$-finite corank 2 double-fold map $(x,y)\mapsto(x^2,y^2,h(x,y))$.
	\end{ex}

	\begin{ex}\label{exSimpleC2C3}All $\cA$-simple germs $\C^2\to \C^3$, classified by Mond in \cite{Mond1985On-the-clasific}, are reflection maps. The families  and the germ
\begin{itemize}
\item[] $S_k\colon\ (x,y)\mapsto(x,y^2,y^3+x^{k+1}y),\, k\geq 0,$
\item[] $B_k\colon\ (x,y)\mapsto(x,y^2,x^{2}y+y^{2k+1}),\, k\geq 2,$
\item[] $C_k\colon\ (x,y)\mapsto(x,y^2,xy^3+x^{k}y),\, k\geq 3,$
\item[] $F_4\colon\ (x,y)\mapsto(x,y^2,x^{3}y+y^5)$
\end{itemize}
are fold maps. The remaining family  consists of the $Z_3$-reflection maps
\begin{itemize}
\item[] $H_k\colon\ (x,y)\mapsto(x,y^3,xy+y^{3k-1}),\, k\geq 1.$
\end{itemize}
 	\end{ex}
		
Further examples of reflection maps, and of maps which does not seem to be reflection maps, can be found in Sections \ref{secAFiniteRefMaps} and \ref{secFinalRemarks}. 
	
	\begin{definition}\label{defCorank}
The \emph{corank of $f\colon X\to Y$ at $x\in X$} is the dimension over $\C$ of the kernel of the differential $df_x$. If $\dim X=n$, then $\corank f_x=n-\rank df_x.$ The corank of a germ $f\colon (\C^n,0)\to (\C^p,0)$ is the corank at $0$ of a representative. We say that $f$ is \emph{singular at $x$} if $df_x$ is singular, that is, if $\corank f_x\geq 1$. Maps with corank $\leq 1$ everywhere are called \emph{curvilinear}.
	\end{definition}

	\begin{lem}\label{lemKerdf}
Let $f$ be a  quasi-reflection map. If $h(x)$ is contained in the facet $C\in\sC$, then $\ker (df_x)=dh_x^{-1}( C^\bot).$
In particular, if $f$ is a germ of reflection map at $x$ and $y=h(x)$, then
\[\corank f=\dim(T_yY\cap C).\]
\begin{proof}
Apply chain rule to $f=\omega \circ h$ and use Proposition \ref{KerDifOrbitMap}.
\end{proof}
	\end{lem}

	\begin{ex}\label{exPerturbationCuspDifferential}
Let $h_s,s\in \C$, be the family of embeddings $t\mapsto (t,t^3-st)$ and $f_s$ the corresponding $Z_2$-reflection maps (Figure \ref{figFoldNormalB}), given by  
\[t\mapsto(t^2,t^3-st).\]
The complex of $Z_2$ is just $\{C_\emptyset,C_1\}$, with $C_\emptyset =\{(x,y)\in \C^2\mid n\neq 0\}$ and $C_1=\{(x,y)\in \C^2\mid x=0\}$. Since $C_\emptyset$ is an open subset of $\C^2$, we have $\C_\emptyset{}^\bot=0$ and, being $h_s$ an embedding, $f_s$ cannot be singular away from $C_1$. The embedding $h_0$ crosses the facet $C_1$ orthogonally and produces the singularity of $f_0$ at the origin. The perturbation $h_s$, with $s\neq 0$, is no longer orthogonal to $C_1$. Consequently, $f_s$ is immersive for $s\neq 0$.
	\begin{figure}
\begin{center}
\includegraphics[scale=1.2]{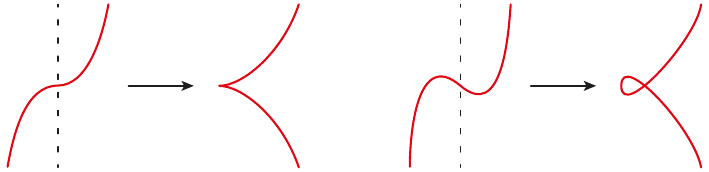}
\end{center}
\caption{The corank of $f$ depends on how $h$ crosses the facets.}
\label{figFoldNormalB}
	\end{figure}
	\end{ex}

	
	\begin{definition}
We say that a germ of reflection map $f\colon (\C^n,0)\to (\C^p,0)$ is \emph{essential} if $\corank f=\rank G$
or, equivalently,
$C^\bot\subseteq T_0Y,$
for the facet $C$ passing through the origin. 
	\end{definition}

	\begin{definition}
Let $H\colon V\to \C^r$ be any map, and let $G$ be a reflection group acting on $V$, with orbit map $\omega$. The \emph{$G$-reflected graph of $H$}  is the map $(\omega, H)\colon V\to \C^{p+r}$, given by
\[x\mapsto (\omega(x), H(x)).\]
The reflected graph $(\omega, H)$ is the reflection map obtained by taking the graph embedding $h$, given by $x\mapsto (x, H(x))$, and letting $G$ act on $V\times\C^{r}$, trivially on the second factor.
	\end{definition}

	\begin{prop}
Every essential reflection map is $\cA$-equivalent to a reflected graph.
\begin{proof}
Take the decomposition $V=C^\bot\oplus C$, where $C\in \sC$ is the facet passing through the origin. If $\rank G=k$, we can choose a basis of$V$ and an isomorphism $\psi\colon V\to \C^p$, making the previous decomposition just $\C^p=\C^k\oplus \C^{p-k}$. We have a new orbit map $\omega'=\psi\circ \omega$ of the form 
$(y_1,\dots,y_p)\mapsto (\omega'_1(y),\dots,\omega'_k(y),y_{k+1}, \dots,y_p).$
 If $f$ is essential, then $T_0Y\cap C^\bot=C^\bot$, and thus $\psi$ can be chosen such that the germ  $\phi\colon (\C^n,0)\to (\C^n,0)$, given by 
 $x\mapsto (h_1(x),\dots,h_n(x)),$
 is a germ of biholomorphism. We have the $\cA$-equivalence $f=\psi^{-1}\circ(\omega'', H)\circ \phi,$  where $H\colon \C^n\to\C^{p-n}$ is $x\mapsto(h_{n+1}\circ \phi^{-1}(x),\dots,h_{p}\circ \phi^{-1}(x)),$
and 
$\omega''\colon (\C^n,0)\to (\C^n,0)$ is  
$y\mapsto (\omega'_1(y),\dots,\omega'_k(y),y_{k+1},\dots,y_{n}).$
\end{proof}
	\end{prop}

%
%
%
%
%
				\section{Corank and injectivity}\label{secCorankAndInj}

We study injectivity of germs of reflection maps, and relate it to their corank. Our main motivation comes from the following Conjecture \ref{conjLe}, due to L\^e \cite{LeConj} (the original statement is different but equivalent to the one we give, as explained in \cite{Fernandez-de-Bobadilla2006A-reformulation}). We show an extended version of this conjecture for germs of reflection maps.

	\begin{conjecture}\label{conjLe}
There is no injective germ $f\colon (\C^2,0)\to (\C^3,0)$ of corank 2.
	\end{conjecture}

	\begin{lem}\label{lemCharInjRefEmb}
A quasi-reflection map $f$ is injective if and only if $h$ is injective and, for all $g\in G\setminus\{1\}$,
\[Y\cap gY\subseteq \Fix  g.\]
\begin{proof}
The map $h$ has to be injective for $f$ to be injective. If $h$ is injective, then $f$ is injective if and only if there exist $y,y'\in Y$, with $y'\neq y$ and $\omega (y')=\omega(y)$, that is, such that $y'=gy$, for some $g\in G$ with $y\notin \Fix  g$.
\end{proof}
	\end{lem}

	\begin{ex}The same perturbation of the cusp as in Example \ref{exPerturbationCuspDifferential} shows how the intersection of $Y$ with its translates $gY$ characterizes injectivity. See Figure \ref{figFoldDoublePoint120width}.
	\begin{figure}
\begin{center}
\includegraphics[scale=1.2]{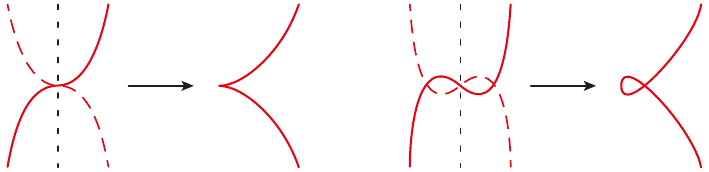}
\end{center}
\caption{Injectivity is characterized by intersections of translates of $Y$.}
\label{figFoldDoublePoint120width}
	\end{figure}
	\end{ex}

	\begin{prop}\label{propInjAndRankG}
There is no injective germ of  quasi-reflection map $f\colon (\C^n,0)\to (\C^p,0)$, with $\rank G>2(p-n)$.

\begin{proof} The hypothesis $\rank G>2(p-n)$ implies that the facet $C\in \sC$ passing through the origin has dimension strictly smaller than $2n-p$. By Proposition \ref{propElementFixesFacet}, there is an element $g\in G$, such that $\Fix  g= C$. Since $Y$ is an $n$-dimensional space passing through the origin, $Y\cap gY$ is a non-empty space of dimension at least $p-2(p-n)=2n-p$. Therefore, $Y\cap gY$ cannot be contained in $\Fix  g$, and the claim follows from Lemma \ref{lemCharInjRefEmb}.
\end{proof}
	\end{prop}

	\begin{prop}\label{propExtendedLeConj}
There is no injective germ of  quasi-reflection map $f\colon (\C^n,0)\to(\C^p,0)$, with $\corank f- \corank h>p-n$.
\begin{proof}
We may assume $p\leq 2n$, since otherwise the statement is  trivial. If $f$ is injective, then $h$ is injective and the dimension of $Y$ is $n$.  If $C\in \sC$ be the facet passing through the origin, then by hypothesis implies $\dim (T_0Y\cap  C^\bot)\geq p-n+1$, and therefore $\dim (Y\cap  C)\leq n-(p-n+1)=2n-p-1$. The same argument as in the proof of Proposition \ref{propInjAndRankG} finishes the proof.
\end{proof}
	\end{prop}

	\begin{cor}\label{corLeConjRefMaps}
There is no injective germ of reflection map $f\colon (\C^n,0)\to (\C^p,0)$, with $\corank f>p-n.$
In particular, Conjecture \ref{conjLe} is true for reflection maps. 
	\end{cor}
	
In view of this result, it is natural to ask the following:
\begin{question*}
Are there injective germs $f\colon(\C^n,0)\to (\C^p,0)$ with $\corank f>p-n$?
\end{question*}
To my knowledge, the only answer so far is for the case of $n=p$, where a map is nonsingular if and only if it is locally injective. Moreover, the previous inequality is sharp, in the sense that we can produce injective germs $f\colon(\C^n,0)\to(\C^p,0)$, with $\corank f=p-n$, (assuming $p\leq 2n$, otherwise the question is trivial).
 Let $r=n-(p-n)$ and $s=(p-n)$, so that $n=r+s$ and $p=r+2s$. Write the coordinates in $\C^n$ as $x=(x_1,\dots,x_r)$ and $y=(y_1,\dots,y_s)$ and take, for $i=1,\dots, s$, pairs $(a_i,b_i)$ of coprime numbers.  It is easy to see that the reflection map 
\[(x,y)\mapsto(x_1,\dots,x_r,y_1{}^{a_1},y_1{}^{b_1},\dots,y_s{}^{a_s},y_s{}^{b_s}),\]
has corank $p-n$ and is injective.

				\section{Normal crossings}\label{secNormalCrossings}
				
	We introduce the notions of stability and $\cA$-finiteness and study normal crossings ---a necessary condition for any finite map to be stable--- for reflection maps. Normal crossings of reflection maps are characterized by two conditions, which can be checked easily in terms of the translates $gY$, with $g\in G$. Studying situations were violating these conditions is unavoidable, we obtain restrictions for normal crossings of reflection maps. Ultimately, these violations are  at the core of the obstructions for stability and $\cA$-finiteness in Section \ref{secObstrAFin}. We give minimal background on normal crossings; a detailed account on the topic (including why normal crossings is necessary for $\cA$-stability) can be found in \cite{Golu-Gui}. 
 
 	\begin{definition}\label{defAStability}
An unfolding $F=(t,f_t(x))\colon (\C^{r+n},\{0\}\times S)\to (\C^{r+p},0)$ of a multi-germ $f=f_0$ is called $\cA$-trivial if it is $\cA$-equivalent to $\id_{\C^r}\times f$ as an unfolding (that is, $\cA$-equivalent via biholomorphisms which are in turn unfoldings of the identity maps on $(\C^n,0)$ and $(\C^p,0)$).

A multi-germ $f\colon(\C^n,S)\to(\C^p,0)$ is \emph{$\cA$-stable} if every unfolding of $f$ is $\cA$-trivial. If, for a small representative of $f$, stability only fails at the preimage of the origin, then we say that $f$ is \emph{$\cA$-finite}. A finite map $f\colon X\to Y$ is $\cA$-stable if, for all $y\in Y$, the multi-germ of $f$ at $f^{-1}(y)$ is $\cA$-stable.
	\end{definition}

Since no equivalence relation other than $\cA$-equivalence appears here and we will not work with global stability, we say just stability instead of (local) $\cA$-stability. In the literature, $\cA$-finite maps are also called $\cA$-finitely determined maps, or maps of finite $\cA$-codimension.
 Some references for stability and $\cA$-finiteness are \cite{GibsonSingularPointsSmoothMappings,Golu-Gui,WallFiniteDeterminacyOfSmoothMapGerms}.

	\begin{notation*}
For any set $X$, we write the small and big diagonals in $X^k$, respectively, as $\Delta(X,k)=\{(x^{(1)},\dots,x^{(k)})\in X^k\mid x^{(i)}=x^{(j)},\text{ for all }i,j\leq k\}$
 and $D(X,k)=\{(x^{(1)},\dots,x^{(k)})\in X^k\mid x^{(i)}=x^{(j)},\text{ for some }i\neq j\leq k\}.$
We also write $\Delta X=\Delta(X,2)$ and $X^{(k)}=X^k\setminus D(X,k)$.
	\end{notation*}
 
	\begin{definition}
A map $f\colon X\to Y$ has \emph{normal crossings} if, for any $k\geq 2$, the restriction of $f\times \dots\times f$ to $X^{(k)}$ is transverse to $\Delta(Y,k)$. The definition extends to multi-germs taking representatives.
	\end{definition}

	\begin{lem}\label{lemImDifOrbitMapTrans}
Let $g\colon V\to V$ be a linear map and $\omega\colon V\to Y$ be a smooth $g$-invariant map. For any $y\in V$ and any vector subspace $W\leq V$, we have $d\omega_y(W)=d\omega_{gy}(gW).$
	\end{lem}

The proof is left to the reader. Now observe that the transversality condition that normal crossings imposes for $k=2$ at a point $(x,x')\in X^{(k)}$, with $f(x)=f(x')=y$, is equivalent to $\im df_x+\im df_{x'}=T_yY$. Putting together the previous result and Lemma \ref{KerDifOrbitMap}, it follows that the bigerm of $\omega$ at $\{y,gy\}$ fails to have normal crossings, for any $y\neq gy\in \sA$. As we shall see, this simple observation is the key to characterize normal crossings of the orbit map.

	\begin{prop}\label{lemNormalCrossingsOrbitMap}
The multi-germs of $\omega$ with normal crossings are exactly
\begin{enumerate}
\item \label{lemNormalCrossingsOrbitMapItem1} monogerms at points $y\in C$, where $C\in \sC$ is a facet of codimension $\leq 1$,
\item \label{lemNormalCrossingsOrbitMapItem2} multi-germs at $\{y,g_2y,\dots, g_ky\}$, with $y\notin \sA$ and $g_2,\dots,g_k$ different elements in $G\setminus\{1\}$.
\end{enumerate}
\begin{proof}
First we discuss the case of multi-germs: By Theorem \ref{thmStabilizerReflection}, there is an open neighborhood $U\subseteq V\setminus \cA$ of $\{y\}$, such that $U\cap g(U)=\emptyset$ for every $g\in G\setminus \{1\}$. For any point $y'\in U$, the fibre of $\omega(y')$ on $U'=U\cup g_2 U\cup\dots\cup g_k U$ is equal to $\{y',g_2y',\dots, g_ky'\}$. Now, from Lemma \ref{KerDifOrbitMap} follows that $\omega$ is a submersion at $y'$ and $g_iy'$, hence $\omega$ has normal crossings on $U'$. That these are the only multi-germs with normal crossings follows immediately from the observation we made just after the previous lemma.


Now we prove the claim for monogerms: If $y\in V\setminus \sA$, then $\omega$ is locally a biholomorphism, and has normal crossings trivially. We claim that on $\sA$, the only other monogerms with normal crossings are centered at points contained in just one reflecting hyperplane. First, observe that any neighborhood $U$ of $y$ can be shrinked, so that the intersection of $U$ with the orbit of any point $y'\in U$ is contained in the orbit of $y'$ by the stabilizer $G_y$ of $y$. If $y$ is contained in just one hyperplane $H$, then we may take $U$ small enough to satisfy $\sA\cap U=C\cap U$, for the facet $C$ containing $y$. By Corollary \ref{corStabilizer}, we have $G_{y'}=G_C=G_{y}$, for all $y'\in \sA\cap U$, and hence the restriction of $\omega$ to $\sA\cap U$ is one-to-one. Therefore, the transversalities to study around $y$ are necessarily of the form (\ref{lemNormalCrossingsOrbitMapItem2}), already checked. Now assume $y$ is contained in two different reflecting hyperplanes  $H_1, H_2$, given by reflections $r_1,r_2\in G$. Let $y'\in U\cap H_1\setminus H_2$. It follows from Theorem \ref{thmStabilizerReflection} that $r_2$ is contained  in the stabilizer $G_y$, but does not fix $y'$, and thus we have $y'\neq r_2y'\in U$. It follows by item (\ref{lemNormalCrossingsOrbitMapItem2}) that $(y',r_2y')$ does not satisfy the required transversality.
\end{proof}
	\end{prop}


Now we study the conditions that an injective map $h$ must satisfy in order to produce a  quasi-reflection map $f$ with normal crossings. As we shall see, $\omega$ is not involved in the computations and the normal crossings condition can be decided just in terms of the translates $gh, g\in G$.  The following two informal examples will provide the intuition: 

	\begin{ex}\label{exNotOrbNormCross}
Let $h\colon X\to \C^2$ parametrize the two left non-dashed curves in Figure \ref{figC2xC2NotOrbitNormalCrossings}. We see that the resulting $Z_{2,2}$-reflection map $f$ does not have normal crossings. However, there is no  apparent failure of transversality when we look just at $h$. The failure becomes clear when we also look at the dashed lines, which are the image of the translate of $h$ by $i_{1,1}\in Z_{2,2}$. 
	\begin{figure}
\begin{center}
\includegraphics[scale=1.2]{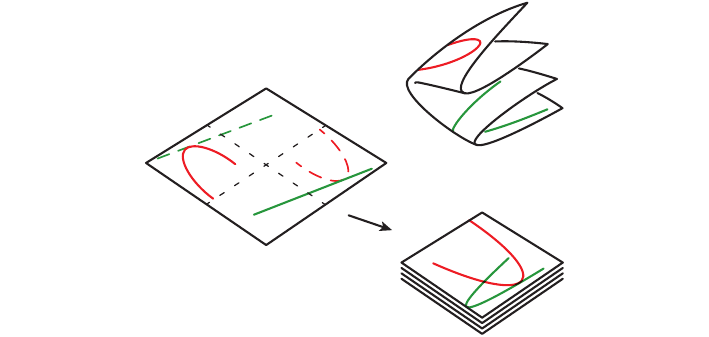}
\end{center}
\caption{Violation of the orbit normal crossings condition.}
\label{figC2xC2NotOrbitNormalCrossings}
	\end{figure}
	\end{ex}

	\begin{definition}\label{defOrbNormCross}
A map $h\colon X\to V$ has \emph{orbit normal crossings} (by $G$) if, for any $k\geq 2$ and any tuple $(g_1,\dots,g_k)\in G^k$, the restriction to $X^{(k)}$ of $g_1h\times \dots\times g_kh$ is transverse to $\Delta(V,k)$. The definition extends to germs by taking representatives.
	\end{definition}

	\begin{ex}\label{exNotOneToOrbit}
Now let $h\colon X\to \C^2$ parametrize the two lines in Figure \ref{figC2xC2NotNormalCrossings}. Again, we obtain a reflection map $f$ where the normal crossings condition fails. However, $h$ has orbit normal crossings, since any two translates cross transversally, and no more than two points are mapped to one by any choice of translates. The problem here is that $Y$ passes through $\sA$ at two points $h(x)=y$ and $h(x')= i_{1,0}y$ in the same orbit. Observe that both the image of $d\omega_y$ and of $d\omega_{{i_{1,0}}y}$ are  the subspace $\{(u_1,u_2)\in \C^2\mid u_1=0\}$. Therefore the images of $df_x$ and $df_{x'}$ cannot span the tangent of $\C^2$, regardless of what the differential of $h$  is at $x$ and $x'$. 
	\begin{figure}
\begin{center}
\includegraphics[scale=1.2]{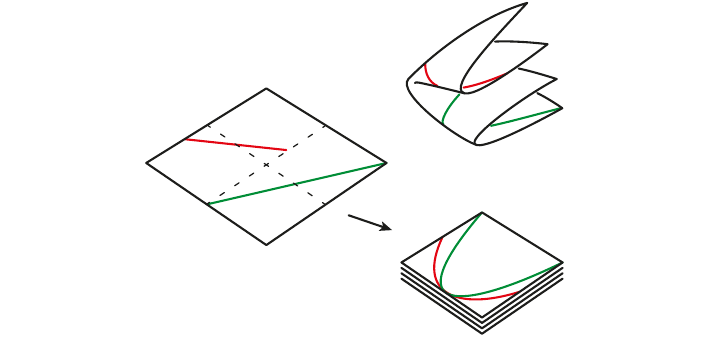}
\end{center}
\caption{Violation of the one-to-orbit over $\sA$ condition.}
\label{figC2xC2NotNormalCrossings}
	\end{figure}
	\end{ex}

	\begin{definition}\label{defOneToOrbOverA}
A map $h\colon X\to V$ is \emph{one-to-orbit over $\sA$} (by $G$) if the restriction of $\omega\circ h$ to $h^{-1}(\sA)$ is one-to-one. The definition extends to germs by taking representatives.
	\end{definition}

	\begin{thm}\label{thmCharacNormalCrossings}
A quasi-reflection map $f$ has normal crossings if and only if $h$ has orbit normal crossings and is one-to-orbit over $\sA$.

\begin{proof}
If $h$ is not one-to-orbit over $\sA$, the proof of item \ref{lemNormalCrossingsOrbitMapItem2} in Proposition \ref{lemNormalCrossingsOrbitMap} shows that $f$ does not have normal crossings, because the required transversality already fails for $\omega$. Therefore $h$ being one-to-orbit over $\sA$ is a necessary condition for $f$ to have normal crossings. Moreover, the condition one-to-orbit over $\sA$ also ensures that transversality to $\Delta(V,k)$ is satisfied by $f\times \dots \times f$ and by $g_1h\times\dots g_k h$ at all points $(x,x^{(2)},\dots,x^{(k)})$ with $h(x)\in \sA$. This is because the condition ensures that the image (by any of the two maps) of such $(x,x^{(2)},\dots,x^{(k)})$ cannot be contained in $\Delta(V,k)$, hence transversality is met trivially. 

This reduces the problem to check that the transversality conditions for $f^k$ and for $g_1h\times\dots g_k h$ are equivalent at points of the form $\underline x=(x,x^{(2)},\dots,x^{(k)})\in X^{(k)}$, such that $h(x)\notin\sA$, and $h(x^{(i)})=g_ih(x)$ for some $g_i\in G$. It follows from Lemma \ref{lemImDifOrbitMapTrans} that at such a point we have
\begin{align*}
(df^k)_{\underline x}T_{\underline x}X^k&=
(d\omega_{h(x)})^k(d(h\times g_2^{-1}h\times \dots\times g_k^{-1}h)_{\underline x}T_{\underline x}X^k).
\end{align*}
Since $h(x)\notin\sA$, it follows that $\omega$ is a local biholomorphism at $h(x)$, and therefore $(d\omega_{h(x)})^k$ is an isomorphism taking the tangent $T_{(h(x),\dots,h(x))}\Delta(Y,k)$ to the tangent $T_{f^{k}(\underline x)}\Delta(Y,k)$. Therefore, we have
\[(df^k)_{\underline x}T_{\underline x}X^k+T_{f^{k}(\underline x)}\Delta(Y,k)=(d\omega_{h(x)})^k\Big(d(h\times g_2^{-1}h\times \dots\times g_k^{-1}h)_{\underline x}T_{\underline x}X^k+T_{(h(x),\dots,h(x))}\Delta(Y,k)\Big),\]
and the claim follows.
\end{proof}
	\end{thm}
	
	\begin{cor}\label{corOnlySmoothBranches}
All multi-germs (monogerms excluded) of reflection maps with normal crossings have only smooth branches.
\begin{proof}Assume $f\colon (\C^n,S)\to (\C^p,y)$ has normal crossings. Since $S$ has more than two elements, the one-to-orbit over $\sA$ condition implies $h(S)\subset V\setminus \sA$. The statement follows because $h$ is an immersion, and $\omega$ is locally  immersive around $h(S)$.
\end{proof}
 	\end{cor}
The following immediate result supports our previous claim, that we are able to decide whether a quasi-reflection map $f$ has normal crossings just by looking at the translates $gh$, with $g\in G$:
	\begin{lem}\label{lemOneToOrbit}
Let $Y\subseteq V$ be the image of $h\colon X\to V$. Then $h$ is one-to-orbit over $\sA$ if and only if
 the restriction of $h$ to $h^{-1}(\sA)$ is injective and, for all $g\in G\setminus \{1\}$, we have 
\[Y\cap gY\cap \sA\subseteq \Fix \,g.\]
\begin{proof}Analogous to the proof of Lemma \ref{lemCharInjRefEmb}.
\end{proof}
	\end{lem}
	
	\begin{ex}\label{exNotOneToOrbit2}
Let $h$ be as in Example \ref{exNotOneToOrbit}. For $g=i_{10}$ and $g=i_{11}$,  we have $Y\cap gY\cap (\sA\setminus \Fix  g)\neq \emptyset$, so $h$ is not one-to-orbit over $\sA$ (see Figure \ref{figNotOneToOrbLinesB}). Is not by chance that the condition fails for two elements of the group: if the one-to-orbit over $\sA$ condition fails for $g$ at $y$ ---that is, if $y\in Y\cap gY \cap (\sA\setminus \Fix  g)$--- then the condition fails for $g'g$ at $y$ as well, for every $g'\in G_y$. 
\begin{figure}
\begin{center}
\includegraphics[scale=1.2]{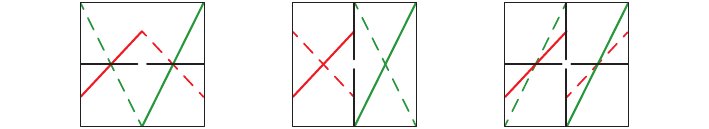}
\end{center}
\caption{$Y, gY$ and $\sA\setminus \Fix  g$, for $g=i_{10},i_{01}$ and $i_{11}$, respectively. Looking at these we can tell whether or not $h$ is one-to-orbit over $\cA$ (see Example \ref{exNotOneToOrbit2})}
\label{figNotOneToOrbLinesB}
\end{figure}
	\end{ex}
	
From this description of the one-to-orbit over $\sA$ condition, we obtain the following obstruction for a reflection map to have normal crossings:

	\begin{prop}\label{propObstrNormalCross}
There is no germ of quasi-reflection map $f\colon (\C^n,0)\to (\C^p,0)$ with normal crossings and $\rank G >2(p-n)+1$.
\begin{proof}
It suffices to show that $h$ is not one-to-orbit over $\sA$. The set $Y\cap gY\cap \sA$ is not empty, because it contains the origin, hence it has dimension at least $2n-p-1$. If we let $\rank G=r$, then the facet $C\in \sC$ containing the origin has dimension $p-r$. By Proposition \ref{propElementFixesFacet}, there exists $g\in G$, such that $C=Fix \ g$. If $r>2(p-n)+1$, then $p-r<p-(2(p-n)+1)=2n-p-1$. Hence, $Y\cap gY\cap \sA$ is not contained in $\Fix \, g$, and the claim follows from Lemma \ref{lemOneToOrbit}.
\end{proof}
	\end{prop}

\begin{rem}
The previous proposition  imposes no restrictions when $p\geq 2n-1$. On the opposite side, it says that only groups with $\rank G=1$ can produce reflection maps $(\C^p,0)\to (\C^p,0)$. In particular, as the orbit map $\omega$ is a reflection map $(\C^p,0)\to (\C^p,0)$, the requirement $\rank G=1$ is consistent with the description of its normal crossings (Proposition \ref{lemNormalCrossingsOrbitMap}), because it implies that there is no facet of codimension $> 1$.
\end{rem}

Putting together the results we have so far and the following classical theorem (essentially due to Whitney \cite{Whitney:1936}, see \cite{Golu-Gui} for a more up-to-date statement) we can characterize stability of reflection maps $f\colon X\to \C^{2n}$, with $n=\dim X$.

	\begin{thm}
A map $f\colon X^n\to Y^{2n}$ is stable if and only if it is an immersion with normal crossings.
	\end{thm}
	
	\begin{cor}
A reflection map $f\colon X^n\to \C^{2n}$  is stable if and only if
\begin{enumerate}
\item\label{EstNTo2N1} $h$ has orbit normal crossings,
\item\label{EstNTo2N2} $Y\cap gY\cap \sA\subseteq \Fix  g$, for all $g\in G\setminus\{1\},$
\item\label{EstNTo2N3} $T_yY\cap  C^\bot=0$, for all $y\in Y$ and the facet $C\in \sC$ containing $y$.
\end{enumerate}
	\end{cor}
	
	\begin{cor}
A germ of reflection map $f\colon (\C^n,0)\to (\C^{2n},0)$ is $\cA$-finite if and only if $h$ satisfies (\ref{EstNTo2N1}) and (\ref{EstNTo2N2}) above, and the following equivalent conditions:
\begin{enumerate}
 \item[(3a)] $T_yY\cap  C^\bot=0$, for all $0\neq y\in Y$ and the facet $C\in \sC$ containing $y$.
 \item[(3b)] $T_0Y\cap  C'^\bot=0$, for all facets $C'$ intersecting $Y$ but not containing the origin. For the facet $C$ containing the origin, $T_yY\cap  C^\bot=0$, for all $0\neq y\in Y\cap C$.
\end{enumerate}
\begin{proof}It is clear that $(1),(2)$ and $(3a)$ characterize $\cA$-finiteness. The equivalence between $(3a)$ and $(3b)$ follows from the fact that transversality is an open condition. Since it holds at $0$, it is satisfied on a small open neighborhood of the origin.
\end{proof}
	\end{cor}

								\section{The double-point space $D^2(\omega)$ of the orbit map}\label{secD2Omega}
In dimensions $n<p<2n$, immersivity and normal crossings are not enough to determine stability and $\cA$-finiteness and some more sophisticated machinery needs to be introduced. For curvilinear maps $f\colon X\to Y$, stability is characterized by the condition of having smooth multiple-point schemes $D^k(f)$ of the \emph{expected dimension} $kn-(k-1)p$ (see \cite{MararMondCorank1} for this result and the definitions of the spaces $D^k(f)$ of corank 1 map-germs). Unfortunately, we are far from having an equivalent result in the presence of corank $\geq2$ points. For reflection maps, however, we only need to deal with double points, and we do it by means of a different double-point scheme denoted by $B^2(f)$. Before this, we introduce the double-point scheme $D^2(f)$, as defined in \cite{MondSomeRemarks}, and describe the scheme $D^2(\omega)$ of an orbit map.

Given $f\colon X\to Y$, let $U\subseteq X$ and $V\subseteq Y$ be coordinate neighborhoods with coordinates $x=x_1,\dots,x_n$ and $y=y_1,\dots,y_p$, respectively. Since the functions $f_j(x')-f_j(x),\, j=1,\dots,p$  vanish on $\{x=x'\}$, we can (after possibly shrinking $U$) find functions $\alpha_{ij}\in \cO_{U\times U}$, such that $f_j(x')-f_j(x)=\sum \alpha _{ji}(x,x')(x'_i-x_i)$. In other words, we can find a $p\times n$ matrix $\alpha$, with entries in $\cO_{U\times U}$, satisfying the equation
\begin{align}
f(x')-f(x)=\alpha(x,x')\cdot(x'-x). \tag{$\dagger$}\label{MatAlpha}
\end{align}

	\begin{definition}
Given $f\colon X\to Y$, the \emph{double-point ideal sheaf} $\sI^2(f)$ in $\cO_{X\times X}$ is, locally,
\[\sI^2(f)=(f\times f)^*\sI_{\Delta(Y,2)}+\langle n\times n \text{ minors of }\alpha\rangle,\]
where $\sI_{\Delta(Y,2)}$ is the ideal sheaf defining the diagonal of $Y$ and $\alpha$ is a matrix satisfying \eqref{MatAlpha}.
 It follows from Cramer's rule that, away from the diagonal, $\sI^2(f)$ is equal to $(f\times f)^*\sI_{\Delta(Y,2)}$.
The \emph{double-point space} $D^2(f)$ is the complex subspace of $X\times X$ cut out by $\sI^2(f)$.
	\end{definition}

We will make use of the following basic properties of $D^2(f)$, which can be found in \cite{Nuno-Ballesteros2015On-multiple-poi} or \cite{Penafort-Sanchis2015Multiple-point-}.
	\begin{prop}\label{PropD2(f)}
For any  $f\colon X\to Y$, with $n=\dim X$ and $p=\dim Y$, the following hold:

\begin{enumerate}
\item[(1)] \label{PropD2(f)item1}As a set, $D^2(f)$ consists of strict double points $(x,x')\in X\times X\text{, with }x'\neq x\text{ and }f(x')=f(x);$
and singular points $(x,x)\in \Delta X\text{, such that $f$ is singular at $x$}.$
\item[(2)] \label{PropD2(f)item2}If $D^2(f)$ is non-empty, then it has dimension $\geq 2n-p$. If $\dim D^2(f)=2n-p$, then $D^2(f)$ is a Cohen-Macaulay space.
\item[(3)] \label{PropD2(f)item3}If $f$ is stable, then $D^2(f)$ is the closure of the strict double points of $f$.
\end{enumerate}
	\end{prop}
								
	\begin{definition}
For any $g\in G\setminus\{1\}$, we write $D_g=\{(y,gy)\in V\times V\}.$
	\end{definition}
	
	 \begin{thm}\label{thmD2OrbitMap}
The double-point space $D^2(\omega)$ of the orbit map is a reduced complex space, with irreducible decomposition  \[D^2(\omega)=\bigcup_{g\in G\setminus\{1\}} D_g.\]
 \begin{proof}
The equality as sets follows from the first item of Proposition \ref{PropD2(f)}, applying Lemma \ref{KerDifOrbitMap},  and Theorem \ref{thmAlphaGeomInv}. Now it suffices to prove that $D^2(\omega)$ is reduced and, since it is a Cohen Macaulay space by  Proposition \ref{PropD2(f)}, it suffices to show that $D^2(\omega)$ is reduced at points $(y,y') \in D^2(\omega)$ with $y\in V\setminus \sA$. For such a point, $\omega$ is not singular at $y$ and thus $y'$ is not equal to $y$. Therefore, $\sI^2(f)$ is generated locally by the coordinate functions of $s(y,y')=\omega(y)-\omega(y')$. Since $ds_{(y,y')}(T_{(y,y')}V\times V)=d\omega_{y}(T_{y}V)+d\omega_{y'}(T_{y'}V)$ and $\omega$ is a submersion, the map $s$ is a submersion, and thus $D^2(\omega)$ is regular at $(y,y')$.  
 \end{proof}
	\end{thm}
	
	\begin{ex}
The double-point space of the orbit map of $Z_{m_1,\dots,m_p}$ (Example \ref{exOrbitMapZm1mp}) decomposes into the branches 
\[D_a=D_{i_a}=\{(y,y')\in\C^p\times\C^p\mid y'_j=e^{\frac{2\pi i}{m_j} a_j}y_j\}.\]  
The case of the Folded Hankerchief is depicted in Figure \ref{figD2FoldedHankerchief}.
\begin{figure}
\begin{center}
\includegraphics[scale=1.2
]{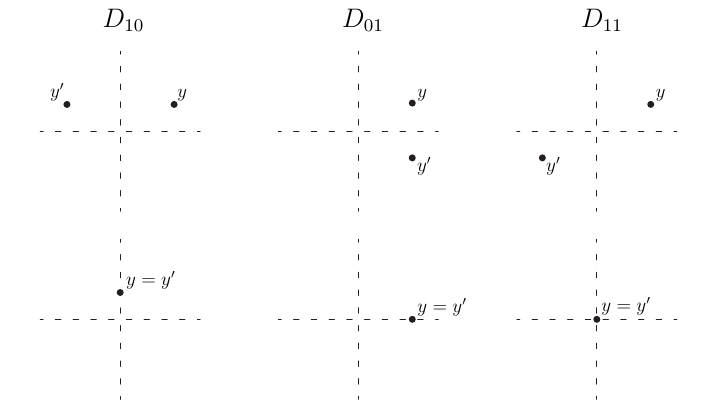}
\end{center}
\caption{Decomposition of $D^2(\omega)$ for the Folded Hankerchief. Strict double points are on the first row, singular ones on the second.}
\label{figD2FoldedHankerchief}
\end{figure}
	\end{ex}
	
Given a reflection map $f$, it is not clear how to obtain the description of the complex structure of the double-point space $D^2(f)$  from $D^2(\omega)$.  We overcome this problem with the introduction the double-point space $B^2(f)$.

				\section{The double point space $B_2(f)$ of a reflection map}\label{secDoubleBlowingUp}
Here we will compute the double-point space $B^2(f)$ of reflection maps. The space $B^2(f)$ of a map $f\colon X\to Y$ was first introduced by Ronga \cite{Ronga1972La-classe-duale}, and then studied  by Laksov \cite{Laksov77ResidualIntersectionsAndTodds}, Kleiman \cite{KleimanMultiplePointFormulasI} and others. Proofs of  the properties in this section can be obtained from these references, and found as stated here in \cite{Penafort-Sanchis2015Multiple-point-}. 

	\begin{notation*}
Given  a vector $g$, whose entries are sections in $\cO_Z$ for some complex space $Z$, we write $\langle g\rangle$ for the ideal sheaf generated by the entries of $g$. If the size of $g$ is $s$, we write $g\wedge u$ for the vector in $\cO_{Z\times {\bf P}^{s-1}}$ whose entries are $u_ig_j-u_jg_i$. 
	\end{notation*}

	\begin{rem}\label{remEqsInCovering}
Any subspace of $X\times {\bf P}^{s-1}$, is  covered by its intersection with the open affine sets 
\[U_i=\{(x,x',u)\in X\times {\bf P}^{s-1}\mid u_i\neq 0\}.\]
On $U_i$, the ideal sheaf $u\wedge g$ is generated by $s-1$ holomorphic functions $g_ju_i-g_iu_j,\, j\neq i.$
Observe also that, if $L$ is a homogeneous function vanishing on $u\in {\bf P}^{s-1}$, for some $(z,u)\in Z\times {\bf P}^{s-1}$, with $u\wedge g(z)=0$, then $L$ vanishes on $g(z)$ as well.
	\end{rem}

We will need the following properties about blowups along closed smooth subspaces.

	\begin{lem}\label{lemCompBlowup}
Let $Y$ ba a codimension $s$ closed smooth subspace of a complex manifold $Z$, whose ideal sheaf is generated by  $g=(g_1,\dots, g_s)$. The blowup of $Z$ along $Y$ is
\[Bl_YZ=\{(z,u)\in Z\times {\bf P}^{s-1}\mid u\wedge g=0\}.\] 
	\end{lem}

	\begin{lem}\label{lemBlowupAsStrictTranf}
If $Y_1,Y_2\subseteq Z$ are closed subspaces, and no irreducible component of $Y_2$ is contained in $Y_1$, then $Bl_{Y_1\cap Y_2}Y_2$ can be realized as the strict transform of $Y_2$ in $Bl_{Y_1}Z$.
	\end{lem}

	\begin{definition}
Given a complex manifold $X$, we write $b\colon B^2(X)\to X\times X$
 for the blowup of $X\times X$ along $\Delta X$. We write the exceptional divisor as $E=b^{-1}(\Delta X).$ In particular, \[B^2(\C^p)=\{(x,x',u)\in \C^p\times\C^p\times{\bf P}^{n-1}\mid u\wedge(x'-x)=0\}.\]

	\end{definition}

\begin{definition}
Given $f\colon X\to Y$, the ideal sheaf $\sH^2(f)$ in $\cO_{B^2(X)}$ is defined, in a small coordinate neighborhood, as 
  \[\sH^2(f)=\langle \alpha(x,x') u\rangle.\]
for any matrix $\alpha$ satisfying equation \eqref{MatAlpha}. The \emph{double-point space} $B^2(f)$ is the subspace of $B^2(X)$ cut out by $\sH^2(f)$.
\end{definition}

	\begin{prop}\label{propOfB2}
For any  $f\colon X\to Y$, with $n=\dim X$ and $p=\dim Y$, the following hold:
\begin{enumerate}
\item \label{propOfB2Item1} As a set, $B^2(f)$ consists of strict double points $z=b^{-1}(x,x')\text{, for some }(x,x')\in X\times X\text{, with }x'\neq x\text{ and }f(x')=f(x)$; and singular points $(x,x,u)\in E\text{, such that }u\in {\bf P}(\ker df_x).$
\item \label{propOfB2Item2} If $B^2(f)$ is non-empty, then it has dimension $\geq 2n-p$. If $\dim B^2(f)=2n-p$, then $B^2(f)$ is locally a complete intersection.
\item \label{propOfB2Item3}The structure map $B^2(X)\to X\times X$ restricts to a morphism of complex spaces 
$b\colon B^2(f)\to D^2(f).$
 This map is proper and surjective, and it is an isomorphism away from the preimage of
$\{(x,x)\in X\times X\mid \corank f_x\geq 2\}.$
\item \label{propOfB2Item4}If $f$ is stable, then $b\colon B^2(f)\to D^2(f)$ is a resolution, and $B^2(f)$ is the closure of its strict double points. In particular, $B^2(f)$ is smooth of dimension $2n-p$.
\end{enumerate}
\end{prop}


				\subsection*{The double-point space $B^2(\omega)$ of the orbit map}\label{secB2Omega}
We obtain a decomposition of $B^2(\omega)$, indexed by the non-trivial elements $g\in G\setminus\{1\}$, analogous to the one for $D^2(\omega)$ in Section \ref{secD2Omega}. 
	\begin{definition}
For any element $g\neq 1$ of a reflection group $G$, we write $b_g\colon B_g\to D_g$
 for the blowup of $D_g$ along $D_g\cap \Delta V=\{(y,y)\in V\times V\mid y\in \Fix g\} .$
	\end{definition}

	\begin{prop}
 \begin{enumerate}
\item \label{propB_gItem1}For all $g\in G\setminus\{1\}$, the space $B_g$ is smooth.

\item \label{propB_gItem2}The structure map $b_g\colon B_g\to D_g$ is an isomorphism if and only if $g$ is a reflection. 	
\item \label{BgStrictTransDg}
$B_g$ embeds in $B^2(\C^p)$ as the strict transform of $D_g$, which is precisely the set of points $(y,gy,[y-gy])\in B_2(V), y\notin \Fix  g$ and points $(y,y,u)\in B^2(V),$ with $y\in \Fix  g$ and  $u\in {\bf P}((\Fix g)^\bot)$.
\end{enumerate}\label{propB_g}
\begin{proof}Item \ref{propB_gItem1} follows because $B_g$ is the blowup of a smooth space along a smooth closed subspace. Item \ref{propB_gItem2} follows  from the universal property of the blowup, because reflections are, by definition, the only finite order linear transformations $g$, such that $D_g\cap \Delta V$  is a hyperplane (and thus a Cartier divisor). Item \ref{BgStrictTransDg} follows from Lemma \ref{lemBlowupAsStrictTranf}, because $D_g$ is not contained in $\Delta V$.
\end{proof}
	\end{prop}

	\begin{definition}
We write $\sH_g$ for the ideal sheaf defining $B_g$ in $B^2(V)$.
	\end{definition}

The following result is a direct consequence Lemma \ref{lemCompBlowup} and Remark \ref{remEqsInCovering}.
	\begin{prop}\label{propGeneratorsBg}
Let $L_1,\dots,L_p$ be independent linear forms $V\to \C$ such that 
$\Fix g=\{y\in V\mid L_{1}(y)=\dots=L_r(y)=0\}.$
The ideal sheaf $\sH_g$ is generated by  $L_1(gy-y'),\dots,L_r(gy-y'),$ and $L_{r+1}(v),\dots,L_n(v).$ The space $B_g$ is cut out by the equations \[y'=gy,\quad L_{r+1}(v)=\dots=L_n(v)=0.\]
	\end{prop}
		

	\begin{ex}
For the group $Z_{2,2}$, the spaces $B_g$  (see Figure \ref{figB2FoldedHankerchief}) are 
\begin{itemize}

\item[]$B_{1,0}=\{((y_1,y_2),(-y_1,y_2),(1:0))\mid (y_1,y_2)\in \C^2\},$
\item[]$B_{0,1}=\{((y_1,y_2),(y_1,-y_2),(0:1))\mid (y_1,y_2)\in \C^2\},$
\item[]$B_{1,1}=\{((y_1,y_2),(-y_1,-y_2),(v_1:v_2))\mid y_1v_2=y_2v_1,\, (y_1,y_2)\in \C^2\}.$
\end{itemize}

\begin{figure}
\begin{center}
\includegraphics[scale=1.2
]{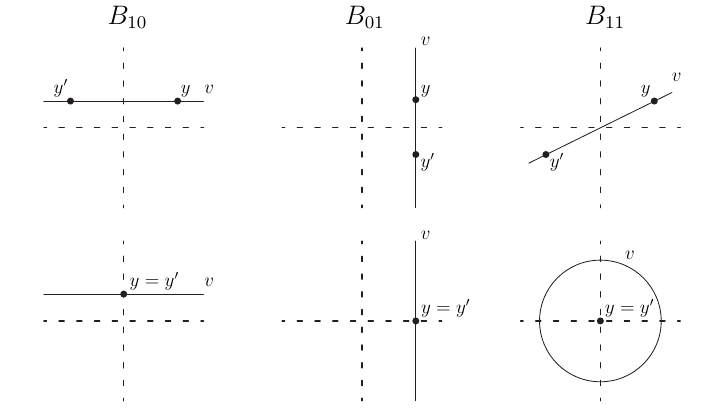}
\end{center}
\caption{Decomposition of $B^2(\omega)$ for the Folded Hankerchief. Observe that $B^2(\omega)$ and $D^2(\omega)$ (Figure \ref{figD2FoldedHankerchief}) fail to be isomorphic precisely at diagonal points of corank $\geq 2$ (see item \ref{propOfB2Item3} of Proposition \ref{propOfB2}).}
\label{figB2FoldedHankerchief}
\end{figure}
	\end{ex}
	
	 \begin{thm}\label{thmB2Omega}
The double-point space $B^2(\omega)$ of the orbit map is a reduced complex space, with irreducible decomposition
 \[B^2(\omega)=\bigcup_{g\in G\setminus\{1\}} B_g.\]
 \begin{proof}We show the equality as sets first: Let $X=\bigcup _{g\in G\setminus\{1\}}B_g$. From Lemma \ref{thmAlphaGeomInv} follows that $B^2(\omega)$ and $X$ coincide away from the exceptional divisor. Thus, $X\subseteq B^2(\omega)$ follows from Lemma \ref{BgStrictTransDg}, because $X$ is a strict transform. For the other inclusion, we need to show that for every vector $v$ in the kernel of $d\omega_x$, there exists an element $g$ of the group which fixes $x$ and such that $v$ is normal to $\Fix g$. This follows directly from Proposition \ref{propElementFixesFacet} and Lemma \ref{KerDifOrbitMap}.

 Now  it suffices to prove that $B^2(\omega)$ is reduced. The equality as sets also implies that the dimension of $B^2(\omega)$ is $p$, and thus the space is a locally complete intersection by  item \ref{propOfB2Item2} of Proposition \ref{propOfB2}. Hence, it suffices to show that $B^2(\omega)$ is generically reduced. In particular, it suffices to show that $B^2(\omega)$ is reduced at all points $(x,x',v) \in B^2(\omega)$ with $x\in V\setminus\sA$. At such a point $B^2(f)$ and $D^2(f)$ are locally isomorphic, and thus the claim follows from Theorem \ref{thmD2OrbitMap}.  
 \end{proof}
	\end{thm}



	\subsection*{From $B^2(\omega)$ to the double points of any reflection map}\label{secB2RefMaps}
In this section we show how to compute the double-point space $B^2(f)$ of a reflection map from the space $B^2(\omega)$ of the orbit map. We will need the following easy lemma:
	\begin{lem}
 Every embedding $h\colon X\hookrightarrow Y$ induces an embedding $\widetilde h\colon B^2(X)\hookrightarrow B^2(Y)$ given as follows: Away from the exceptional divisor $E_X$, set $\widetilde h=b_Y{}^{-1}\circ(h\times h)\circ b_X,$
  where $b_X$ and $b_Y$ are the structure maps of the corresponding blowups. Locally around $E_X$, $\widetilde h$ is given by 
  \[(x,x,u)\mapsto (h(x),h(x),\alpha(x,x')\cdot u),\]
 for any matrix $\alpha$ satisfying \eqref{MatAlpha} for $h$. 
	\end{lem}

	\begin{prop}
Let $f=F\circ h\colon X\to Z$, for some embedding $h\colon X\hookrightarrow Y$ and some map $F\colon Y\to Z$. Then $B^2(f)=\widetilde h^{-1}(B^2(F)).$
\begin{proof}Away from the diagonal the statement is obvious. Around diagonal points, let $h(x)-h(x')=\alpha(x,x')(x-x')$ and $F(y)-F(y')=\beta(y,y')(y-y')$. Observe that the construction satisfies the chain rule $f(x)-f(x')=\beta(h(x),h(x'))\cdot(\alpha(x,x')(x-x')).$
Since $\widetilde h$ is given by $(x,x',u)\mapsto (h(x),h(x'),\alpha(x-x')\cdot u)$, the equality $\sH^2(f)=\widetilde h^*(\sH^2(F))$ follows by construction.
 \end{proof}
	\end{prop}

	\begin{definition}
For any $g\in G\setminus\{1\}$ and any embedding $h\colon X\to V$, we write $B_g(h)=\widetilde h^{-1}(B_g)$ and $\sH_g(h)=(\widetilde h)^*\sH_g.$
	\end{definition}

	\begin{cor}\label{corDoublePointRefMapDecomp}
The double points of any reflection map $f$ decompose as 
\[B^2(f)=\bigcup_{g\in G\setminus\{1\}}B_g(h).\]
	\end{cor}

	\begin{ex}\label{exCompuBg(h)}
Let $h\colon \C^n\to V$ be an injective linear map, and let $L_1,\dots,L_p$ be independent linear forms on $V$, such that $\Fix \ g=\{x\in V\mid L_{1}(x)=\dots=L_r(x)=0\}.$
 From Proposition \ref{propGeneratorsBg}, it follows that the ideal sheaf $\sH_g(h)$ is generated by $L_1(gh(x)-h(x')),\dots, L_r(gh(x)-h(x')),$
and $L_{r+1}(h(u)),\dots,L_p(h(u)).$
Equivalently, the space $B_g(h)$ is cut out by the equations
\[h(x')=gh(x),\quad L_{r+1}(h(u))=0,\dots,L_n(h(u))=0.\]
	\end{ex}

	
						\section{Obstructions to $\cA$-finiteness}\label{secObstrAFin}
 The following space will be crucial to describe stability and $\cA$-finiteness of reflection maps:
	\begin{definition}
For any reflection group $G$, we write $S_G$ for the singular locus of $B^2(\omega)$.
	\end{definition}

	\begin{prop}\label{propDescriptionSG}
The space $S_G$ is the union of all intersections $B_g\cap B_{g'}$, with $g,g'\in G\setminus\{1\}, \, g'\neq g$. It consists of the points
\begin{enumerate}
\item $(y,gy,[y-gy])$, with $y\in \sA\setminus \Fix  g$,
\item $(y,y,v)$, $y\in \Fix g\cap \Fix g'$, $v\in {\bf P}((\Fix g)^\bot\cap (\Fix g')^\bot)$, with $g,g'\in G\setminus\{1\}, \, g'\neq g$.
\end{enumerate}

\begin{proof}The first statement follows simply because $B^2(\omega)$ is the union of the smooth spaces $B_g$ (Proposition \ref{propB_g}, item \ref{propB_gItem1}). To show the second statement, take $y\in V$. If $y\notin \sA$ then only $g$ takes $y$ to $gy$, because the pointwise stabilizer of $G$ at $y$ is trivial (Corollary \ref{corStabilizer}). In the case $y\in \Fix g$ the second item is just the expression of a point $(y,y,v)\in B_g\cap B_{g'}$. If $y\in \sA\setminus \Fix g$, then the pointwise stabilizer of $G$ at $y$ is not trivial, so there is some $g'\neq g$, such that $g'y=gy$. Therefore, we have $(y,gy,[y-gy])\in B_g\cap B_{g'}$.
\end{proof}
	\end{prop}
	
	\begin{cor}\label{corPointsInSG}The space $S_G$ contains the following points:
\begin{enumerate}
\item \label{corPointsInSGItem1}$(y,y,v)$, with $y\in C_1, v\in{\bf P}( C_2{}^\bot)$, for two different facets $C_1,C_2\in \sC$, satisfying $C_1\subseteq\langle C_2\rangle$.
\item \label{corPointsInSGItem2}$(y,y,v)$, with $y\in H,\,v\in {\bf P}(H^\bot)$, where $H$ is the reflecting hyperplane of any reflection $r\in G$ of order $\geq 3$. 
\end{enumerate}
\begin{proof}
For item \ref{corPointsInSGItem1}, take $g_1,g_2\in G$ with $\Fix g_i=\langle C_i\rangle$.  For item \ref{corPointsInSGItem2}, take $r$ and $r^2$. The result follows because $r^2$ is also a reflection with reflecting hyperplane $H$.
\end{proof}

	\end{cor}

%

	\begin{lem}\label{lemDoublePointsStableRefMap}
Let $f\colon X^n\to \C^p$ be a stable reflection map. Then
 the spaces $B_g(h)$ are disjoint smooth spaces of dimension $2n-p$, and  $\widetilde h^{-1}(S_G)=\emptyset$.

\begin{proof}
Since $B^2(f)$ is smooth by Proposition \ref{propOfB2}, and $B^2(f)=\bigcup B_g(h)$, it suffices to show that any two different branches $B_g(f)$ and $B_{g'}(f)$ cannot have an irreducible component in common. Assume there is one such component $Z$. Again by item \ref{propOfB2Item4} of Proposition \ref{propOfB2}, $Z$ contains a strict point of the form $b^{-1}(x,x')$, with $h(x')=gh(x)=g'h(x)\neq h(x)$. This cannot happen if $h(x)\notin \sA$, because the pointwise stabilizer of $G$ at $h(x)$ is trivial (Corollary \ref{corStabilizer}). On the other hand, $h(x)\in \sA$ implies that $f$ is not one-to-orbit over $\sA$, and thus $f$ does not have normal crossings, and hence stability fails.
The second item follows immediately from the first, because $B_g(h)=\widetilde h^{-1}(B_g)$, and $S_G$ is union of all intersections of the form $B_g\cap B_{g'}$ (Proposition \ref{propDescriptionSG}).
\end{proof}
	\end{lem}

	\begin{thm}\label{thmNoRMStabCor2}
There is no stable germ of reflection map of corank $\geq 2$, and all stable essential reflection maps of corank one are fold maps.
\begin{proof}
Let $f\colon (\C^n,0)\to (\C^p,0)$ be a corank 2 germ of reflection map, or a singular germ of $Z_m$-reflection map, with $m\geq 3$. By Lemma \ref{lemDoublePointsStableRefMap}, it suffices to show that $\widetilde h^{-1}(S_G)$ is not empty. Let $C_1\in \sC$ be the facet passing through the origin. Since $f$ is singular, then intersection of $T_0Y$ and $C_1$ is not trivial, and the case of corank $1$ follows from Corollary \ref{corPointsInSG}, item \ref{corPointsInSGItem2}.  The corank 2 hypothesis implies $\dim (T_0Y\cap C_1{}^\bot)\geq 2$. Since the codimension of $C_1$ is at least $2$, there is a facet $C_2$ with $\dim C_2=\dim C_1+1$ and $C_1\subseteq \langle C_2\rangle$. This implies $\dim(T_0Y\cap  C_2 {}^\bot)\geq 1$. For any  $v\in {\bf P}(T_0Y\cap C_2 {}^\bot)$, we have a point $(0,0,v)\in \im \widetilde h\cap S_G$, by Corollary \ref{corPointsInSG}, item \ref{corPointsInSGItem1}. 
\end{proof}
	\end{thm}


	\begin{thm}\label{thmNoRMAFinCor2Dims}
For $p<2n-1$, there is no $\cA$-finite germ of reflection map $f\colon (\C^n,0)\to (\C^p,0)$ of $\corank f\geq 2$, and all $\cA$-finite germs of essential reflection maps of corank one are fold maps.
\begin{proof}Let $\pi\colon B^2(f)\to X$ be the composition of $b\colon B^2(f)\to D^2(f)$ and the first projection $X\times X\to X$. If $\pi(h^{-1}(S_G))$ has dimension greater than 0, then the locus where $f$ is unstable is not issolated, and hence $f$ is not $\cA$-finite. Assume that there is a counter example $f$ to the statement.

If $f$ is a corank one germ of $Z_m$-reflection map, with $m\geq 3$, then $\widetilde h^{-1}(S_G)\neq \emptyset$ by Corollary \ref{corPointsInSG}, item \ref{corPointsInSGItem2}, and thus $\widetilde h^{-1}(S_G)$ has dimension $\geq 2n-(p+1)>0$.  The result follows because $\pi$ is finite for curvilinear maps.

If $f$ is a germ of reflection map with $\corank f\geq 2$ the argument is more subtle, because $\pi$ is not finite. If $C_0\in \sC$ is the facet containing the origin, then $\dim(T_0Y\cap C_0^\bot)\geq 2$. Since $1\leq 2n-p$, there is a facet $C\in \sC$, not passing through the origin, and such that $1\leq \dim(T_0Y\cap C^\bot)\leq 2n-p$.  Take $g\in G$, such that $\Fix g=\langle C\rangle$. On one hand, $1\leq \dim(T_0Y\cap C^\bot)$ implies that $B_g(h)$ is not empty, and hence it has dimension at least $2n-p$. On the other hand, since $\dim B_g(h)\geq2n-p$, the condition $\dim(T_0Y\cap C^\bot)\leq 2n-p$ implies that $B_g(h)$ cannot be contained in the fiber of the origin. Since $f$ is stable away from the origin, it follows from item \ref{propOfB2Item4} of Proposition \ref{propOfB2} that $B_g(h)$ contains a component $Z$, having dimension $2n-p$ and consisting generically on strict double points.
 
  Let $Z'=\pi(Z)\subseteq X$. Since $f$ is finite,  only finitely many strict double points can be mapped to the same point in $Z'$, and hence $\dim Z'=2n-p$. Since $C$ does not contain the origin, $\sA\setminus C$ does so, and thus $h^{-1}(\sA\setminus C)$ is a locally analytic subspace of codimension at most $1$ passing through the origin. Since $\dim Z'=2n-p$, the space $W=Z'\cap h^{-1}(\sA\setminus C)$ has dimension at least $1$. We claim that all points in $Z$ mapped to $W$ are contained in $\widetilde h^{-1}(S_G)$. The claim implies $\dim \pi(h^{-1}(S_G))> 0$, and hence contradicts the stability of $f$ away from the origin.

To show the claim, take any point $z\in Z$, mapping to a point $x\in X$ with $h(x) \in\sA\setminus C$. If $h(x)\notin \Fix g$, then $\widetilde h(z)$ is a point of the form $(y,gy,[gy-y])$, contained in $S_G$ by Proposition \ref{propDescriptionSG}. If $h(x)\in \Fix g$, then $\widetilde h(z)=(y,y,v)$ with $y=h(x)$ and $v\in {\bf P}(C^\bot)$. Since $y=h(x)\notin C$, there is a facet $C'\neq C$, such that $y\in C' \subseteq \langle C\rangle$. By Corollary \ref{corPointsInSG}, we have $\widetilde h(z)\in S_G$. 
\end{proof}
	\end{thm}

							\section{$\cA$-finite reflection maps}\label{secAFiniteRefMaps}
	
We show $\cA$-finiteness criteria for dimensions $p\geq 2n-1$ and use them to produce families of $\cA$-finite reflection maps. Our criteria are consequences of the following general criteria, which can be obtained easily by combining Marar-Mond criterion of stability in corank one, Mather-Gaffney criterion of $\cA$-finiteness and the relation between $B^2(f)$ and $D^2(f)$ discussed in Proposition \ref{propOfB2}. 

	\begin{prop}\label{propCharAFinPgeq2N}
For $p\geq 2n$, a germ $f\colon (\C^n,0)\to (\C^p,0)$ is $\cA$-finite if and only if $B^2(f)$ is contained in the fiber of the origin.
	\end{prop}

	\begin{prop}\label{propAfiniteNTo2N-1}
A germ $f\colon (\C^n,0)\to (\C^{2n-1},0)$ is $\cA$-finite if and only if the following conditions hold:
\begin{enumerate}
 \item Away from the fiber of the origin, $B^2(f)$ is a reduced curve, 
\item The germ $f$ has no strict triple points.
 \end{enumerate}

 	\end{prop}

	\begin{cor}\label{corCharAFinRefMapPgeq2N}
For $p\geq 2n$, a germ of reflection map $f\colon (\C^n,0)\to(\C^{2n},0)$ is $\cA$ finite if and only if all $B_g(h)$ are contained in the fiber of the origin.
	\end{cor}

	\begin{cor}\label{corAfiniteRefMapNTo2N-1}
A germ of reflection map $f\colon (\C^n,0)\to(\C^{2n-1},0)$ is $\cA$-finite if and only if, away from the fiber of the origin, all $B_g(h)$ are empty or disjoint reduced curves without intersection.
	\end{cor}
	
\begin{caution*}In what follows, whenever we say that some numbers, say $m_1,\dots,m_s,r_1,\dots,r_t$, are pairwise coprime, we mean that every $m_i$ and $m_j$, with $i\neq j$; every $r_i$ and $r_j$, with $i\neq j$; and every $m_i$ and $r_j$ are coprime numbers.
\end{caution*}
		
	\begin{thm}
Let $H$ be an $n\times n$ matrix with entries in $\Q$, all whose submatrices have maximal rank. Let $m_1,\dots,m_{n},r_{1},\dots,r_{n}$ be pairwise coprime positive integers. The following germ $f\colon (\C^n,0)\to (\C^{2n},0)$ is $\cA$-finite:
  \[x\mapsto ({x_1}^{m_1},\dots,{x_n}^{m_n},{(Hx)_1}^{r_{1}},\dots,{(Hx)_{n}}^{r_{n}}).\]
\begin{proof}The claim follows directly from Proposition \ref{propExNTo2N} and Lemma \ref{lemP1CoprimeImpliesP2} below.
\end{proof}
	\end{thm}

	\begin{ex}\label{exNTo2N}
For pairwise coprime positive integers $m_i$ and $r_j$, the following germs are $\cA$-finite:
\begin{align*}
x&\mapsto(x^{m_1},x^{r_1}),\\
(x_1,x_2)&\mapsto(x_1{}^{m_1},x_2{}^{m_2},(x_1+x_2)^{r_1},(x_1-x_2)^{r_2}),\\
(x_1,x_2,x_3)&\mapsto(x_1{}^{m_1},x_2{}^{m_2},x_3{}^{m_3},(x_1+x_2+x_3)^{r_1},(x_1-x_2+2x_3)^{r_2},(x_1+2x_2-x_3)^{r_3}).
\end{align*}
	\end{ex}

	\begin{thm}
Let $H$ be an $(n-1)\times n$ matrix with entries in $\Q$, all whose submatrices have maximal rank. Let $m_1,\dots,m_{n},r_{1},\dots,r_{n-1}$ be pairwise coprime positive integers, with all $r_j$ being odd. The following germ $f\colon (\C^n,0)\to (\C^{2n-1},0)$ is $\cA$-finite:
\[x\mapsto ({x_1}^{m_1},\dots,{x_n}^{m_n},{(Hx)_1}^{r_{1}},\dots,{(Hx)_{n-1}}^{r_{n-1}}):\]
\begin{proof}It follows from Proposition \ref{propExNTo2Nminus1} and Lemma \ref{lemP1CoprimeImpliesP3AndP4} below.
\end{proof}
	\end{thm}

	\begin{ex}\label{exNTo2NMinus1}
For  pairwise coprime positive integers $m_i$ and $r_j$, with $r_j$ odd, the following germs are $\cA$-finite:
\begin{align*}
(x_1,x_2)&\mapsto(x_1{}^{m_1},x_2{}^{m_2},(x_1+x_2)^{r_1}),\\
(x_1,x_2,x_3)&\mapsto(x_1{}^{m_1},x_2{}^{m_2},x_3{}^{m_3},(x_1+x_2+x_3)^{r_1},(x_1-x_2+2x_3)^{r_2}).
\end{align*}
	\end{ex}

	\begin{rem}\label{propExC^2C3}
It is not clear to me to what extent the condition that the exponents $r_i$ are odd is necessary. For example, it can be proven easily that Lemma \ref{exNTo2NMinus1} does not require the mentioned condition for $n=2$. In particular, the map 
\[(x,y)\mapsto(x^{m_1},y^{m_2},(x+y)^{r_1})\]
is $\cA$-finite, for any pairwise coprime $m_1,m_2$ and $r_1$. 
	\end{rem}

	\begin{definition}
We say that a matrix $H$, with entries in $\C$, \emph{has the property {\bf C1}} if all its submatrices have maximal rank. 
	\end{definition}

	\begin{definition}
Let $H$ be an $n\times n$ matrix with complex entries and let $m_1,\dots,m_{n},r_1,\dots,r_n$ be positive integers.  We say that $H$ \emph{has the property {\bf C2} for $m_i$ and $r_j$}  if, for any $\xi_1,\dots,\xi_n,\eta_1,\dots,\eta_{n}\in \C$, satisfying 
\[\xi_i{}^{m_i}=1,\quad \eta_i{}^{r_{i}}=1\]
 and 
 \[det(H_{ij}(\xi_j-\eta_i))=0,\]
  there are at least $n$ of the numbers $\xi_1,\dots,\xi_n,\eta_1,\dots,\eta_{n}$ which are equal to 1.
	\end{definition}

	\begin{prop}\label{propExNTo2N}
Let $m_1,\dots,m_{n},r_1,\dots,r_n$ be positive integers and let $H$ be a matrix satisfying {\bf C1} and {\bf C2} for $m_i$. The following germ $f\colon (\C^n,0)\to (\C^{2n},0)$ is $\cA$-finite: 
\[x\mapsto ({x_1}^{m_1},\dots,{x_n}^{m_n},{(Hx)_1}^{r_{1}},\dots,{(Hx)_n}^{r_{n}}).\]
\begin{proof}
The map $f$ is the reflection map given by the graph $h$ of $H$ and $G=Z_{m_1,\dots,m_{n},r_1,\dots,r_n}$. By Corollary \ref{corCharAFinRefMapPgeq2N}, it suffices to show that, for every  $i_a\in G\setminus\{1\},$ with $a=(a_1,\dots,a_{2n})$, the double-point branch $B_{i_a}(h)$ is contained in the fiber of the origin.  For $j=1,\dots,n$, let 
\[\xi_j=e^{\frac{2\pi i }{m_j}a_j}\text{ and }\eta_j=e^{\frac{2\pi i}{m_{n+j}}a_{n+j}}.\]
By Example \ref{exCompuBg(h)}, the equations of $B_{i_a}(h)$ in the variables $x,x',u$ of $B^2(\C^n)$ are 
\begin{align*}
x'_i&=\xi_ix,\\
(Hx')_i&=\eta_i(Hx)_i,
\end{align*}
together with 
\begin{align*}
u_i=0\,\text{ if }\,\xi_i=1,\\
(Hu)_i=0\,\text{ if }\,\eta_i=1.
\end{align*}
Eliminating $x'$, the first two equations turn into the homogeneous linear system $Mx=0$, where $M_{ij}=H_{ij}(\xi_j-\eta_i)$. If $B_{i_a}(h)$ is not contained in the origin, then $\det(M)=0$ and, by {\bf C2}, at least $n$ of the numbers $\xi_1,\dots,\xi_n, \eta_1,\dots,\eta_n$ are equal to 1. Let $H'$ be the $2n\times n$ matrix obtained by adding the rows of $H$ to the identity matrix of size $n$. Clearly, {\bf C1} implies that every $n\times n$ submatrix of $H'$ has rank $n$. If we pick any $n$ of the elements in $\eta_i$ or $\xi_i$ which are equal to 1, we obtain a system $S u=0$ for the corresponding submatrix $S$ of $H'$. Since $S$ has maximal rank, this equation cannot be met by any $u\in {\bf P}^{n-1}$.
\end{proof}

	\end{prop}

	\begin{definition}\label{defP3}
Let $H$ be an $(n-1)\times n$ matrix and let $m_1,\dots,m_{n},r_1,\dots,r_{n-1}$ be positive integers. We say that $H$ \emph{has the property {\bf C3} for $m_i$ and $r_i$} if, for any $\xi_1,\dots,\xi_n,\eta_1,\dots,\eta_{n-1}\in \C$, satisfying 
\[\xi_i{}^{m_i}=1,\quad \eta_i{}^{r_{i}}=1\] and \[\rank(H_{ij}(\xi_j-\eta_i))<n-1,\]
 there are at least $n$ of the numbers $\xi_1,\dots,\xi_n,\eta_1,\dots,\eta_{n-1}$ which are equal to 1.
	\end{definition}

	\begin{definition}\label{defP4}
Let $H$ be an $(n-1)\times n$ matrix and let $m_1,\dots,m_{n},r_1,\dots,r_{n-1}$ be positive integers. We say that $H$ \emph{has the property {\bf C4} for $m_i$ and $r_i$} if, for any complex numbers 
\[\xi_1,\dots,\xi_n,\quad \eta_1,\dots,\eta_{n-1},\quad \xi'_1,\dots,\xi'_n,\quad \eta'_1,\dots,\eta'_{n-1},\]
 satisfying 
\[\xi_j{}^{m_j}=\xi'_j{}^{m_j}=1,\quad\eta_i{}^{r_{i}}=\eta'_i{}^{r_i}=1,\]
the following holds: If the $(3n-2)\times n$ matrix $M$, obtained by concatenating the $n\times n$ diagonal matrix with entries $\xi_i-\xi'_i$ and the matrices $(H_{ij}(\xi_j-\eta_i))$ and $(H_{ij}(\xi'_j-\eta'_i))$, satisfies 
\[\rank (M)<n,\]
 then there are at least $n$ indices
$1\leq i_1<\dots<i_s\leq n$ and $1\leq j_{s+1}<\dots<j_{n}\leq n-1$, such that $\xi_{i_k}=1$ or $\xi'_{i_k}=1$, and $\eta_{j_k}=1$ or $\eta'_{j_k}=1$.
 	\end{definition}

	\begin{prop}
\label{propExNTo2Nminus1}
Let $m_1,\dots,m_{n},r_1,\dots,r_{n-1}$ be positive integers and let $H$ be a matrix satisfying {\bf C1}, {\bf C3} and {\bf C4} for $m_i$ and $r_j$. The following germ $f\colon (\C^n,0)\to (\C^{2n-1},0)$ is $\cA$-finite:
\[x\mapsto ({x_1}^{m_1},\dots,{x_n}^{m_n},{(Hx)_1}^{r_{1}},\dots,{(Hx)_{n-1}}^{r_{n-1}})\]
\begin{proof}The map $f$ is the $G$-reflection map given by the graph $h\colon \C^n\hookrightarrow \C^{2n-1}$ of $H$ and $G=Z_{m_1,\dots,m_{n},r_1,\dots,r_{n-1}}$. By Corollary \ref{corAfiniteRefMapNTo2N-1}, it suffices to check the following conditions:
\begin{enumerate}
\item Let  $i_a\in G\setminus\{1\},$ with $a=(a_1,\dots,a_{2n-1})$. Away from the fiber of the origin, the double-point branch $B_{i_a}(h)$ is a empty or a reduced curve.
\item The intersection of any two double-point branches $B_{i_a}(h)\cap B_{i_{a'}}(h)$, with $a\neq a'$, is contained in the fibre of the origin.
\end{enumerate}
The equations of $B_{i_a}(f)$ are as in the proof of \ref{propExNTo2N} and the same ideas used there show that {\bf C3} implies (1). To show (2), let 
\[\xi_1=e^{\frac{2\pi i}{m_{1}}a_{1}}\,,\dots,\  \xi_n=e^{\frac{2\pi i}{m_{n}}a_{n}}\,,\quad \eta_1=e^{\frac{2\pi i }{r_{1}}a_{n+1}},\dots,\ \eta_{n-1}=e^{\frac{2\pi i }{r_{n-1}}a_{2n-1}}\]
and 
\[\xi'_1=e^{\frac{2\pi i}{m_{1}}a'_{1}}\,,\dots,\  \xi'_n=e^{\frac{2\pi i}{m_{n}}a'_{n}}\,,\quad \eta'_1=e^{\frac{2\pi i }{r_{1}}a'_{n+1}},\dots,\ \eta'_{n-1}=e^{\frac{2\pi i }{r_{n-1}}a'_{2n-1}}.\]
After eliminating $x'$, a point in space $B_{i_a}(h)\cap B_{i_a'}(h)$ corresponds to a point $(x,u)\in \C^n\times {\bf P}^{n-1}$, satisfying 
\[Mx=0,\]
 with $M$ as in the definition of {\bf C4}, and the conditions
\begin{align*}
u_i=0\,\text{ if }\,\xi_i=1\text{ or }\xi'_i=1,\\
(Hu)_i=0\,\text{ if }\,\eta_i=1\text{ or }\eta'_i=1.
\end{align*}
Therefore, condition {\bf C4} implies (2).
\end{proof}
	\end{prop}

To give examples having arbitrarily high exponents, we should to check condition {\bf C2}, or conditions {\bf C3} and {\bf C4}, for all these exponents. The following two technical lemmas spare us to do so under some hypothesis. Their proofs are given in the final Section \ref{subsecAuxResults}.

	\begin{lem}\label{lemP1CoprimeImpliesP2}
Let $m_1,\dots,m_{n},r_1,\dots,r_n$ be pairwise coprime positive integers  and let $H$ be an $n\times n$ matrix with entries in $\Q$. If $H$ satisfies {\bf C1}, then it also satisfies {\bf C2} for $m_i$ and $r_j$.
	\end{lem}

	\begin{lem}\label{lemP1CoprimeImpliesP3AndP4}
Let $m_1,\dots,m_{n},r_1,\dots,r_{n-1}$ be pairwise coprime positive integers, with all $r_j$ being odd,  and let $H$ be an $(n-1)\times n$ matrix with entries in $\Q$. If $H$ satisfies {\bf C1}, then it also satisfies {\bf C3} and {\bf C4} for $m_i$ and $r_j$.
	\end{lem}

\section{Final Remarks}\label{secFinalRemarks}
We look for reflection maps in some well known classifications of map-germs. As we review these classifications, we see that in dimensions where the $\cA$-finiteness of reflection maps is unobstructed, that is, whenever $p$ is at least $2n-1$, much of the families we come across are reflection maps, specially in the $\cA$-simple case. We also give some motivation for the search of $\cA$-finite maps of any corank with coordinate functions of unbound multiplicity.

\subsection*{Corank one, classifications} There are several classifications  $\cA$-finite maps that we can search for reflection maps. All of them consist only of corank one germs and they include maps $(\C^n,0)\to (\C^p,0)$, for $(n,p)$ equal to $(1,2),(1,3),(2,3),(3,4)$ and $(n,2n)$.  

For dimensions $(n,p)=(1,2)$ the $\cA$-simple plane curves were classified by Bruce and Gaffney \cite{Bruce:1982} . For $(n,p)=(1,3)$, the classification of $\cA$-simple space curves is due to Gibson and Hobbs \cite{Gibson:1993}. It is not  surprising that all these curves are reflection maps, as it is enough that one of the coordinate functions of a curve is a monomial $t^m$ for the curve to be a $Z_m$-reflection map. Perhaps more interesting is the fact that all these curves can be described using just $Z_m$-reflection maps with $m=2,3,4$. It is also worth noting that most of these map-germs have two or three monomials as coordinate functions. Thus, they can be described as $Z_{m_1,m_2}$ or $Z_{m_1,m_2,m_3}$-reflection maps defined by very simple embeddings $h\colon (\C,0)\hookrightarrow (\C^p,0)$. For example, the second list contains the germs of the form $t\mapsto (t^3,t^{3k+1},t^{3n+2})$, which are $Z_{3,3k+1,3n+2}$-reflection maps given by the embedding $t\mapsto (t,t,t)$.

  We have already mentioned Mond's classification of $\cA$-simple germs $(\C^2,0)\to (\C^3,0)$, consisting only of fold maps and $Z_3$-reflection maps \cite{Mond1985On-the-clasific} (see Example \ref{exSimpleC2C3}). Also in \cite{Mond1985On-the-clasific}, Mond includes the list of all $\cA$-finite non-simple germs up to $\cA_e$-codimension 6. This list contains further reflection maps, namely the $Z_4$-reflection map  $T_4\colon\,(x,y)\mapsto(x,y^4,xy+y^3),$ and the $Z_3$-reflection map $X_4\colon\,(x,y)\mapsto(x,y^3,x^2y+xy^2 +y^4).$
The remaining germs in the list are not reflection maps in any obvious way, but most of them can be expressed as unfoldings of reflection maps of the form $y\mapsto (y^a,y^b)$.

The classification of $\cA$-simple germs $(\C^n,0)\to (\C^{2n},0), n\geq 2$, was carried out in \cite{Klotz:2007} by C. Klotz, O. Pop and J. H. Rieger. It consists of two lists, one for $n=2$, whose families are labelled by roman numerals, and another for $n\geq 3$, labelled by arabic numerals. The list for $n=2$ consists only of reflection maps: it contains
fold-maps ($\text{I}_k,\text{II}_k,\text{III}_{k,l}, \text{IV}_k,\text{V}$ and $\text{VI}$),
$Z_3$-reflection maps ($\text{VII}_k,\text{VIII}_k,\text{IX}_k,\text{X},\text{XI},\text{XIII}$ and $\text{XIV}$) and one family ($\text{XII}_k$) of $Z_4$-reflection maps. A number of these germs can also be regarded as further $Z_m$ and (non-essential) $Z_{m_1,m_2}$-reflection maps: $I_k$ (groups $Z_{2k+1}$ and $Z_{2,2k+1}$), $\text{II}_k$ ($Z_3$ and $Z_{2,3}$), $\text{V}$ ($Z_5$ and $Z_{2,5}$), $\text{X}$ ($Z_4$ and $Z_{3,4}$) and $\text{XI}$ ($Z_5$ and $Z_{3,5}$). For $n\geq 3$ the situation is similar, with the exception that there is a family, labelled $22_k$ (with $k\geq 2$, for $n=3$, and just one germ for $n\geq 4$, given by $k=2$), which does not seem to be made of reflection-maps. It has the form $(x_1,\dots,x_{n-1},y)\mapsto(x_1,\dots,x_{n-1},x_1y+y^3,x_2y,\dots,x_{n-1}y,x_1y^{2}+y^{2k+1},x_2y^2+y^4)$. 

Away from the dimensions $(n,p),p\geq n$, there is Houston and Kirk's \cite{Houston:1999} list of map-germs $(\C^3,0)\to (\C^4,0)$, containing all $\cA$-simple germs and all the non-simple ones with $\cA_e$-codimension at most $4$.
 Many of these are fold maps, labelled $A_k,D_k,E_6,E_7,E_8,B_k,C_k$ and $F_k$.
 By Theorem \ref{thmNoRMStabCor2}, the list cannot contain $Z_m$-reflection maps for $m\geq 3$ (observe that, a priori, it could contain non-essential corank 1 reflection maps). Being germs of corank 1, all germs in the list are unfoldings of maps $(\C^2,0)\to(\C^3,0)$ and, remarkably, all the $\cA$-simple germs which are not fold maps are unfoldings of $Z_3$-reflection maps. Some can be seen also as unfoldings of $Z_m$-reflection maps, for $m=4,5$ and $7$. The list of non-simple includes no reflection map, but some of its members are unfoldings of $Z_3$ and $Z_4$-reflection maps.
 
\subsection*{Families with unbound multiplicity and Mond's conjecture.}
 In the case of corank greater than one, we are not much interested in classifications, but in problems such as Conjecture \ref{conjLe} or Mond's conjecture \cite{Mond:1991}:
 
\begin{conjecture} The image Milnor number of an $\cA$-finite germ $f\colon(\C^n,0)\to(\C^{n+1},0)$ is greater than or equal to its $\cA_e$-codimension, with equality in the quasi-homogeneous case.
\end{conjecture}

The concepts in the statement are beyond the scope of this work, but the conjecture may be regarded as the result for maps corresponding to the inequality $\tau\leq \mu$ for the Tjurina and Milnor numbers of a germ of hypersurface with an isolated singularity. Mond's conjecture is known to be true for $n=1$ \cite{Mond:1995} and $n=2$ \cite{Jong:1991} and (in a slightly different form) for fold maps \cite{Houston:1998}, but all other cases remain open. However, in \cite{Fernandez-de-Bobadilla:2016} it is shown that the conjecture can be reduced to families of examples. More precisely, in order to prove the conjecture for a fixed dimension $n$ and up to corank $k$, it suffices to check the conjecture for a family of $\cA$-finite germs 
\[
f^N\colon (\C^n,0)\to(\C^{n+1},0),\  N\in \N,
\]
where at least $k+1$ coordinate functions of $f^N$ start at order at least $N$ (that is, the degree of the lowest non-zero term in the Taylor expansion is at least $N$).
However, even in the case of corank one, there is no such a family in the literature for dimensions were Mond's conjecture remains open. This makes the search of $\cA$-finite examples with coordinate functions of unbounded multiplicity into an interesting problem.

Recently, Sharland \cite{Altintas:2014} has introduced examples of $\cA$-finite germs $(\C^3,0)\to (\C^4,0)$ of corank 2. From Theorem \ref{thmNoRMStabCor2} it follows that these cannot be reflection maps. However, Sharland's family
\[B_{2l+1}\colon\ (x,y,z)\mapsto(x,y^2+xz,z^2+\alpha xy,y^{2l+1}+y^{2l}z+yz^{2l}-z^{2l-1})\]
(the values of $\alpha$ are determined by $l$) consists of unfoldings of double fold maps
\[(y,z)\mapsto(y^2,z^2,y^{2l+1}+y^{2l}z+yz^{2l}-z^{2l-1}).\]
In the same vein as in Sharland's work, one could produce $\cA$-finite examples $(\C^n,0)\to(\C^{n+1})$ of corank 2, with increasing multiplicities, by taking suitable unfoldings of the $\cA$-finite germ 
\[(x,y)\mapsto(x^{m_1},y^{m_2},(x+y)^{m_3}),\]
with $m_1,m_2$ and $m_3$ pairwise coprime (see Proposition \ref{propExC^2C3}).

 Another approach to obtain $\cA$-finite map-germs with high multiplicity and corank in dimensions $p<2n-1$ is to consider singular maps obtained from reflection groups in a less rigid way. For instance, a ``generalized reflection map'' can be defined as a composite 
 \[X\stackrel{h}{\hookrightarrow}V\stackrel{\omega}{\longrightarrow}\C^m\stackrel{L}{\longrightarrow}\C^p,\] where $h$ is an embedding, $\omega$ is the orbit map of a reflection group and $L$ is a submersion. This allows us to increase $m$ for fixed $n$ and $p$, avoiding certain obstructions related to the otherwise unavoidable bad intersections of the translates of $Y$. This line of work is being developed in colaboration with Bruna Or\'efice Okamoto and Jo\~ao Nivaldo Tomazella

\subsection*{Invariants of reflection maps:}Along this paper we have shown how involved questions like injectivity, normal crossing and stability, are captured easily by the translates of $Y$ and the complex $\sC$ for reflection maps. We have used this machinery to answer general questions about reflection maps, such as the obstructions to produce injective or $\cA$-finite reflection maps in certain dimensions and corank. It seems to me that this approach could be taken further, and that a deeper study of the invariants of a reflection maps could be made. Here are some ideas:

The invariants of double-fold maps $\C^2\to \C^3$ were studied in \cite{Penafort-Sanchis2014THE-GEOMETRY-OF}. Recall that the group $Z_{2,2}$ is generated by the reflections $i_{1,0}$ and $i_{0,1}$. The space $B^2(f)$ projects to $X=\C^2$ and its different branches $B_g(h)$ project to spaces that in \cite{Penafort-Sanchis2014THE-GEOMETRY-OF} were written as $D_g(f)\subseteq \C^2$. The contact of these branches with the facets in $\sC$ determines the singularities. For example, wherever $D_{i_{1,0}}(f)$ crosses transversally $\Fix i_{1,0}\setminus \Fix i_{0,1}$ we find a cross-cap. Also, wherever $D_{i_{1,0}}(f)$ crosses transversally  $\Fix i_{0,1}\setminus \Fix i_{1,0}$ a ``standard self tangency'' appears. Bearing this in mind, one can compute in a very easy way the number of cross-caps found after perturbing generically the function that defines the reflected graph, and the same for the number of standard self tangencies. It should be possible to generalize this to groups other than $Z_{2,2}$ and dimensions other than $\C^2\to \C^3$.

Studying the action of the group $\cA$ on reflection maps is another direction of work. For any reflection group, there is an adapted contact equivalence $\cK^\omega$, so that two reflected graphs $(\omega, H)$ and $(\omega, H')$ are $\cA$-equivalent if $H$ and $H'$ are $\cK^\omega$-equivalent  (see \cite{Penafort-Sanchis2014THE-GEOMETRY-OF} for details). It still not clear whether the $\cK^\omega$-equivalence is also a necessary condition for the $\cA$-equivalence, nor how to extend this contact equivalence to reflection maps which are not reflected graphs. We also would like to know if the $\cA_e$-codimension or the degree of $\cA$-determinacy can be computed or estimated in an easier way from the composite structure of reflection maps, or from corresponding notions for $\cK^\omega$-equivalence.

Another open question is the recognition of reflection maps: Given a map $f$, how can we decide whether or not it is a reflection map? From which groups can we obtain $f$? We know, for instance, that multi-germs of reflection maps with normal crossings have only smooth branches (Corollary \ref{corOnlySmoothBranches}). If a map shows a normal crossings multi-germ consisting of, say, a cross-cap intersecting a regular branch, then this map is not a reflection map. Of course this is a very weak result, and so far we have only been able to say that certain maps, such as the family $22_k$ in the previous section about classifications, do not seem to be reflection maps.

\section{Auxiliary results:}\label{subsecAuxResults}

Here we prove Lemma \ref{lemP1CoprimeImpliesP2} and Lemma \ref{lemP1CoprimeImpliesP3AndP4}. First we need to fix some notation. For any subset $J\subseteq\{1,\dots,n\}$, write $x^J=\prod_{i\in J}x_i.$
Fixed a positive degree $d\leq n$, we write $\cS=\{J\subseteq\{1,\dots,n\}\mid \vert J\vert=d\}$. Observe that a polynomial $p=\sum_{J\in \cS} a_Jx^J$, with all $a_J\neq 0$, is just a homogeneous polynomial having a nonzero term for each monomial $x_{i_1}\cdots\, x_{i_d}$, with $i_1<\dots<i_d$, and all other terms equal to zero.
As usual, we write $\Q(\xi_1,\dots,\xi_s)$ for the extension of $\Q$ obtained by adjoining the elements $\xi_1,\dots,\xi_s\in \C$.
	\begin{lem}\label{lemA}
Let $ \xi_1,\dots,\xi_s\in \C$, such that $\xi_i{}^{m_i}=1$, for some pairwise coprime positive integers $m_1,\dots, m_s$. If $\xi_1\notin \Q$, then $\xi_1\notin\Q(\xi_2,\dots,\xi_s)$. 
	\end{lem}

	\begin{lem}\label{lemB}
Let $p\in\Q[x_1,\dots,x_n]$ be of the form $p=\sum_{J\in \cS}a_Jx^J,$
 with all $a_J\neq 0$. If $p(\xi)=0$, for some $\xi=(\xi_1,\dots,\xi_n)\in \C^n$, then there are at least $n-d+1$ coordinates $\xi_i$, such that $\xi_i\in \Q(\xi_1,\dots,\widehat{\xi_{i}},\dots,\xi_n)$. 
 \begin{proof}
 We proceed by induction on $d$. For $d=1$, the equality $p(\xi)=a_1\xi_1+\dots +a_n\xi_n=0,$ with $0\neq a_i\in \Q$, implies  
 $\xi_i\in \Q(\xi_1,\dots,\widehat{\xi_{i}},\dots,\xi_n)$ for all $i$, as desired.
 
Now assume that the statement is true for polynomials of degree $d-1$. If there is one $\xi_i\notin \Q(\xi_1,\dots,\widehat{\xi_{i}},\dots,\xi_n)$, then take the decomposition $p=x_iq+r$, where
  $q=\sum_{i\in J\in \cS}a_Jx^{J\setminus\{i\}}$ and $r=\sum_{i\notin J\in \cS}a_Jx^J$
do  not involve $x_i$. From the assumption that $\xi_i\notin \Q(\xi_1,\dots,\widehat{\xi_{i}},\dots,\xi_n)$ and the equality $0=\xi_iq(\xi_1,\dots,\widehat{\xi_{i}},\dots,\xi_n)+r(\xi_1,\dots,\widehat{\xi_{i}},\dots,\xi_n)$, it follows $q(\xi_1,\dots,\widehat{\xi_{i}},\dots,\xi_n)=0$.
Now observe that $q$ is a  homogeneous polynomial of degree $d-1$, whose nonzero terms correspond to all degree $d-1$ monomials in the variables $x_1,\dots,\widehat{x_{i}},\dots,x_n$ with no repeated factors. Therefore, the induction hypothesis applies to $q$ and thus at least $(n-1)-(d-1)+1=n-d+1$ of the numbers $\xi_j,$ with $j\neq i$, are contained in $\Q(\xi_1,\dots,\widehat{\xi_{i}},\dots,\xi_n)$. 
 \end{proof}
	\end{lem}

	\begin{notation*}
Given two matrices $A$ and $B$ of the same size, we write $A\bullet B$ for the entrywise product of $A$ and $B$.
	\end{notation*}

	\begin{lem}\label{lemNumberOfReals}
Let $H=(H_{ij})$ be a $n\times n$ matrix with entries in $\Q$, all whose submatrices have maximal rank. 
 Let $\xi_1,\dots,\xi_n,\eta_1,\dots,\eta_n\in \C$, such that $\xi_i{}^{m_i}=1$ and $\eta_i{}^{r_{i}}=1$, for some  positive integers $m_i$ and $r_j$, such that every two numbers in the collection $m_1,\dots,m_n,r_1,\dots,r_n$ are coprime. If the matrix  
 \[M=(H_{ij}(\xi_j-\eta_i))\]
satisfies $\det (M)=0$, then at least $n+1$ of the numbers $\xi_1,\dots,\xi_n,\eta_1,\dots,\eta_n$ are contained in $\Q$.
\begin{proof}
 By interchanging every $\xi_i$ and $\eta_i$ if necessary, we may assume that $r_{1},\dots,r_{n}$ are not divisible by $2$. Thus if $\eta_i$ is rational, it is not equal to $-1$ and hence must be equal to $1$. This allows us to divide up the $\eta_i$ into some which are not rational and the others which are equal to $1$. After reordering the $\eta_i$ and the $\xi_i$, we may assume that $\eta_1,\dots,\eta_k\notin \Q$ and $\eta_{k+1}=\dots=\eta_n=1$. Under this assumptions, the matrix $M$ is $H\bullet A$, with 
\[A=
\left(\begin{array}{cccc}
\xi_1-\eta_1		& 	 		& \xi_n-\eta_1 \\
\vdots 	&  			& \vdots \\

\xi_1-\eta_k		& 	 		& \xi_n-\eta_k \\
\xi_1-1		& 	 		& \xi_n-1 \\
\vdots 	&  			& \vdots \\
\xi_1-1 		& \cdots	  		& \xi_n-1
\end{array}\right).\] 

 Expanding by the first row of $M$, we have
$\det(M)=\det(H\bullet A^{(1)}) \eta_ 1+ B^{(1)}$, for some $B^{(1)}\in \Q(\xi_1,\dots,\xi_n,\eta_2,\dots,\eta_n)$ and
\[A^{(1)}=
\left(\begin{array}{cccc}
1		& \cdots	 		& 1 \\

\xi_1-\eta_2		& 	\cdots 		& \xi_n -\eta_2\\
\vdots 	&  			& \vdots \\

\xi_1-\eta_k		& \cdots	 		& \xi_n-\eta_k \\
\xi_1-1		& \cdots	 		& \xi_n-1 \\
\vdots 	&  			& \vdots \\
\xi_1-1 		& \cdots	  		& \xi_n-1
\end{array}\right).\] 
 Since $\eta_1$ is not rational, from Lemma \ref{lemA} follows that $\eta_1\notin \Q(\xi_1,\dots,\xi_n,\eta_2,\dots,\eta_n)$. Therefore, $\det(M)=0$ implies $\det(M^{(1)})=0,\text{ with }M^{(1)}=H\bullet A^{(1)}.$
 
The same argument applies to the matrix $M^{(1)}$ and $\eta_2\notin \Q$ (expanding the determinant by the second row) and so on so forth until we conclude the following: The condition that $\det(M)=0$ implies $\det(M^{(k)})=0$, for the matrix $M^{(k)}=H\bullet A^{(k)}$, with 
\[A^{(k)}=\left(\begin{array}{ccccc}
1 						& \cdots		& 1	 \\
\vdots 	  		& 	&\vdots  \\

1 						& \cdots		& 1	 \\
\xi_1-1 		 	&  \cdots		& \xi_n-1 \\
\vdots 	  		& 	&\vdots  \\
\xi_1-1 		 	&  \cdots		& \xi_n-1

\end{array}\right),\]where each $\xi_i$ is repeated $n-k$ times. In the case $k=n$ we are done, because $M^{(k)}=H$ and therefore $\det(M)=0$ is in contradiction with the hypothesis that $\det(H)=0$.
 
 If $k<n$, then we set $d=n-k\geq 1$. For any $J\in \cS$, let $a_J$ be the product of the minor of $H$ obtained by picking the columns in $J$ and the last $d$ rows, and the minor obtained by picking the columns of $H$ not in $J$ and the first $k$ rows, with the corresponding sign. We have that 
 \[\det(M^{(k)})=p(\xi_1-1,\dots,\xi_n-1)\]
  for the polynomial $p(x)=\sum_{J\in \cS}a_Jx^J=0$. The hypothesis that every submatrix of $H$ has maximal rank implies that non of the $a_J$ with $J\in \cS$ equals zero.
If $\det(M^{(k)})=0$, then from Lemma \ref{lemB} follows that there are at least $k+1$ of the numbers $\xi_i$, such that 
\[\xi_i-1\in \Q(\xi_1-1,\dots, \widehat{\xi_{i}-1},\dots,\xi_n-1)\]
 or, equivalently, $\xi_i\in \Q(\xi_1,\dots,\widehat{\xi_{i}},\dots,\xi_n).$ By Lemma \ref{lemA}, these  $\xi_i$ are in $\Q$. Since we also have $\eta_{k+1}=\dots=\eta_{n}=1$, the result follows.
\end{proof}
	\end{lem}

	\begin{proof}[Proof of Lemma \ref{lemP1CoprimeImpliesP2}:]
From Lemma \ref{lemNumberOfReals}, there are $n+1$ real numbers in the collection of $\xi_i$ and $\eta_j$. These  have to be either $1$ or $-1$ but, since only one of the numbers $m_i$ and $n_j$ can be divisible by $2$, at most one  of them is $-1$.
	\end{proof}

	\begin{lem}
Let $m_1,\dots,m_{n},r_{1},\dots,r_{n-1}$ pairwise coprime positive integers, with all $r_i$ odd,  and let $H$ be an $(n-1)\times n$ matrix with entries in $\Q$. If $H$ satisfies {\bf C1}, then it  satisfies {\bf C3} and {\bf C4} for $m_i$.
	\end{lem}

	\begin{proof}[Proof of Lemma \ref{lemP1CoprimeImpliesP3AndP4}:]
 To show that $H$ satisfies {\bf C3}, set $r_{n}=1$ and add a new row to $H$, with entries in $\Q$, so that the resulting matrix $H'$  satisfies {\bf C1}.  Assume that $\xi_1,\dots,\xi_n,\eta_1,\dots,\eta_{n-1}\in \C$, with $\xi_i{}^{m_i}=1\text{ and }\eta_i{}^{r_{i}}=1,$ satisfy  $\rank(H_{ij}(\xi_j-\eta_i))<n-1$ and set $\eta_n=1$. Since $\det (H'_{ij}(\xi_j-\eta_i))=0$, the claim follows immediately from Lemma \ref{lemNumberOfReals}.

Now we show that $H$ satisfies {\bf C4} for $m_i$. Given a solution $\xi_j,\xi'_j,\eta_i,\eta'_i$ as in the definition of {\bf C4}, by reordering the matrices if necessary, we may assume $\xi_1=\xi'_1,\dots, \xi_k=\xi'_k$ and $\xi_{k+1}\neq\xi'_{k+1},\dots,\xi_n\neq\xi'_n$, for some $0\leq k\leq n$. We divide the proof in the cases $k=0$,  $0<k<n$ and  $k=n$. If $k=0$, the upper $n\times n$ submatrix of $M$ has rank $n$ and {\bf C4} is satisfied trivially.

If  $0<k<n$, then $M$ has the form
\[
M=L\bullet \left(
 \begin{array}{cccccc}
&&\huge{\text{0}}&&
 \\ \cline{4-6}
 &&\multicolumn{1}{c|}{}&\xi_{k+1}-\xi'_{k+1} &&0\\  
  &
     &\multicolumn{1}{c|}{}&&\ddots \\
  &&\multicolumn{1}{c|}{}&0&&\xi_n-\xi'_n \\ \hline
  \xi_1-\eta_1&\dots&\multicolumn{1}{c|}{\xi_k-\eta_1}&& \\ 
  \vdots&&\multicolumn{1}{c|}{\vdots}&&
  \\
    \xi_1-\eta_{n-1}&\dots&\multicolumn{1}{c|}{\xi_k-\eta_{n-1}}&& \\ \cline{1-3}
    &
      &&\huge{\text{*}}\\
    
 \end{array}
\right)
,
\]
where $L$ is the matrix obtained by concatenating the identity matrix and two copies of $H$. The condition that $\rank(M)<n$ implies the vanishing of every minor of size $k$ of the matrix
\[ 
M'=L'\bullet \left(\begin{array}{ccc}
  \xi_1-\eta_1&\dots&\xi_k-\eta_1 \\
  \vdots&&\vdots\\
    \xi_1-\eta_{n-1}&\dots&\xi_k-\eta_{n-1}
\end{array}
\right),\]
for the corresponding submatrix $L'$ of $L$. By reordering rows, we may assume
$\eta_1\neq 1,\dots,\eta_l\neq 1$ and $\eta_{l+1}=\dots=\eta_{n-1}=1$. Let $L''$ be the upper $k\times k$ submatrix of $L'$. The matrix $L''$ satisfies the hypothesis of Lemma \ref{lemNumberOfReals} trivially (for $n=k$) and the matrix $M''=(L''_{ij}(\eta_i-\xi_j))$ satisfies $\det(M'')=0$. If $s$ is the number of elements $\xi_i,$ with $i\leq k,$ which are equal to 1, then it suffices to show $n\leq s+n-1-l$, that is, to show $s\geq l+1$. This follows by applying Lemma \ref{lemNumberOfReals} to the matrix $L''$.

Finally, assume $k=n$. The condition that $\rank(M)<n$ is equivalent to $\rank (M')<n$, for the matrix
\[ 
M'=L'\bullet \left(\begin{array}{ccc}
  \xi_1-\eta_1&\dots&\xi_n-\eta_1 \\
  \vdots&&\vdots\\
    \xi_1-\eta_{n-1}&\dots&\xi_n-\eta_{n-1}\\
  \xi_1-\eta'_1&\dots&\xi_n-\eta'_1 \\
  \vdots&&\vdots\\
    \xi_1-\eta'_{n-1}&\dots&\xi_n-\eta'_{n-1}
\end{array}
\right),\]
where $L'$ is obtained by concatenating two copies of $H$. Since $(\xi_1,\dots,\xi_n,\eta_1,\dots,\eta_{n-1})\neq (\xi'_1,\dots,\xi'_n,\eta'_1,\dots,\eta'_{n-1})$, and $\xi_j=\xi'_j$ for all $j$, it follows that $\eta_i\neq \eta'_i$ for some $i$. By placing the $i$th row in the first position, we may assume $\eta_1\neq\eta'_1$. The condition that $\rank(M)<n$ implies the vanishing of the determinant of the matrix
\[
N =
\left(\begin{array}{c}
H_1\\
\vdots\\
H_{n-1}\\
H_1
\end{array}\right)\bullet \left(\begin{array}{ccc}
  \xi_1-\eta_1&\dots&\xi_n-\eta_1 \\
  \vdots&&\vdots\\
    \xi_1-\eta_{n-1}&\dots&\xi_n-\eta_{n-1}\\
  \xi_1-\eta'_1&\dots&\xi_n-\eta'_1
  \end{array}
\right)
,\]
where $H_i$ stands for the $i$th row of $H$. Expanding the first and last rows, an easy computation shows that $\det(N)=0$ implies $\det(N')=0$, with $N'=H'\bullet A'$ and 
\[
H'= 
\left(\begin{array}{c}
H_1\\
\vdots\\
H_{n-1}\\
H_1
\end{array}\right),\quad         
A' =\left(\begin{array}{ccc}
1&\dots&1\\
  \xi_1-\eta_2&\dots&\xi_n-\eta_2 \\
  \vdots&&\vdots\\
    \xi_1-\eta_{n-1}&\dots&\xi_n-\eta_{n-1}\\
  \xi_1-1&\dots&\xi_n-1
  \end{array}
\right).\]
Since all $n_r$ are odd, we may reorder rows further, so that $\eta_2,\dots,\eta_k\notin \Q$ and $\eta_{k+1},\dots,\eta_{n-1}=1$. The proof is finished by following the same steps as in the proof of Lemma \ref{lemNumberOfReals}, with $A'$ and $H'$ playing the roles of $A^{(1)}$ and $H$, respectively.
	\end{proof}

\bibliography{MyBibliography} 

\begin{thebibliography}{10}

\bibitem{Altintas:2014}
A.~Alt{\i}nta\c{s}~Sharland.
\newblock Examples of finitely determined map-germs of corank {$2$} from
  $n$-space to {$(n+1)$}-space.
\newblock {\em Internat. J. Math}, page~17, 2014.

\bibitem{AltintasThesis}
A.~Altintas~Sharland.
\newblock {\em Multiple point spaces and finitely determined map-germs}.
\newblock PhD thesis, University of Warwick, UK, 2011.

\bibitem{Arnold:1978}
V.I. Arnol'd.
\newblock Critical points of functions on manifolds with boundary.
\newblock {\em Russian Math. Surveys}, 33:99--116, 1978.

\bibitem{Berczi:2012}
G.~B\'erczi and A.~Szenes.
\newblock Thom polynomials of {M}orin singularities.
\newblock {\em Annals of Mathematics}, 175:567--629, 2012.

\bibitem{Bruce:1982}
J.~W. Bruce and T.~Gaffney.
\newblock Simple singularities of mappings {$\mathbb C,0\to\mathbb C^2,0$}.
\newblock {\em J. London Math. Soc.}, 26(2):465--474, 1982.

\bibitem{Chevalley:1955}
C.~Chevalley.
\newblock Invariants of finite groups generated by reflections.
\newblock {\em American Journal of Mathematics}, 77(4):778--782, 1955.

\bibitem{Jong:1991}
T.~de~Jong and D~van Straten.
\newblock {\em Disentanglements}, volume 1462 of {\em Lecture Notes in Math.
  Singularity theory and its applications, Part I (Coventry, 1988/1989).}
\newblock Springer, Berlin, 1991.

\bibitem{Feher2006On-the-second-o}
L.M. Feh{\'e}r and B.~K{\"o}m{\"u}ves.
\newblock On the second order {T}hom-{B}oardman singularities.
\newblock {\em Fundamenta Mathematica}, 191:249--264, 2006.

\bibitem{Feher2012Thom-series-of-}
L.M. Feh{\'e}r and R.~Rim{\'a}nyi.
\newblock Thom series of contact singularities.
\newblock {\em Annals of Mathematics}, 176(3):1381--1426, 2012.

\bibitem{Fernandez-de-Bobadilla2006A-reformulation}
J.~Fern\'andez~de Bobadilla.
\newblock A reformulation of {L}\^e's conjecture.
\newblock {\em Indag. Mathem.}, 17(3):345--352, 2006.

\bibitem{Fernandez-de-Bobadilla:2016}
J.~Fern\'andez~de Bobadilla, J.J. Nu{\~n}o~Ballesteros, and
  G.~Pe{\~n}afort~Sanchis.
\newblock A {J}acobian module for disentanglements and applications to {M}ond's
  conjecture.
\newblock arXiv:1604.02422, 2016.

\bibitem{Gaffney1993Polar-multiplic}
T.~Gaffney.
\newblock Polar multipliciteis and equisingularity of map germs.
\newblock {\em Topology}, 32(1):185--223, 1993.

\bibitem{GibsonSingularPointsSmoothMappings}
C.~G. Gibson.
\newblock {\em Singular points of smooth mappings}.
\newblock Pitman, 1979.

\bibitem{Gibson:1993}
C.~G. Gibson and C.~Hobbs.
\newblock Simple singularities of space curves.
\newblock {\em Math. Proc. Cambridge Phil. Soc.}, 113:297--310, 1993.

\bibitem{Golu-Gui}
M.~Golubitsky and V.~Guillemin.
\newblock {\em Stable Mappings and Their Singularities}.
\newblock Springer, 1986.

\bibitem{GoryunovMond}
V.V. Goryunov and D.~Mond.
\newblock Vanishing cohomology of singularities of mappings.
\newblock {\em Compositio Mathematica}, 89(1):45--80, 1993.

\bibitem{Hilbert:1890}
D.~Hilbert.
\newblock Ueber die theorie der algebraischen formen.
\newblock {\em Math. Ann.}, 36(4):473--534, 1890.

\bibitem{HoustonTop}
K.~Houston.
\newblock Local topology of images of finite complex-analytic maps.
\newblock {\em Topology}, 36(5):1077--1121, 1997.

\bibitem{Houston:1998}
K.~Houston.
\newblock On singularities of folding maps and augmentations.
\newblock {\em Math. Scand.}, 82:191--206, 1998.

\bibitem{Houston:1999}
K.~Houston and N.~Kirk.
\newblock On the classification and geometry of corank 1 map-germs from
  three-space to four-space.
\newblock {\em London Math. Soc. Lecture Note Ser.}, pages 325--351, 1999.

\bibitem{Humphreys1990Reflection-Grou}
J.~Humphreys.
\newblock {\em Reflection Groups and {C}oxeter Groups}, volume~29.
\newblock Cambridge studies in advanced mathematics, 1990.

\bibitem{KleimanMultiplePointFormulasI}
S.~L. Kleiman.
\newblock Multiple-point formulas {I}: Iteration.
\newblock {\em Acta mathematica}, 147:13--49, 1981.

\bibitem{Klotz:2007}
C.~Klotz, O.~Pop, and J.~H. Rieger.
\newblock Real double-points of deformations of {$\mathcal A$}-simple map-germs
  from {$\mathbb R^n$} to {$\mathbb R^{2n}$}.
\newblock {\em Math. Proc. Cambridge Phil. Soc.}, 142:341--363, 2007.

\bibitem{Laksov77ResidualIntersectionsAndTodds}
D.~Laksov.
\newblock Residual intersections and {T}odds formula for the double locus of a
  morphism.
\newblock {\em Acta mathematica}, 140(1-2):75--92, 1977.

\bibitem{LehrerTaylorUniRefGroups}
Gustav~I. Lehrer and Donald~E. Taylor.
\newblock {\em Unitary Reflection Groups}.
\newblock Cambridge, 2009.

\bibitem{Marangell2010The-General-Qua}
R.~Marangell and R.~Rim{\'a}nyi.
\newblock The general quadruple point formula.
\newblock {\em American Journal of Mathematics}, 132(4):867--896, 2010.

\bibitem{MararMondCorank1}
W.~L. Marar and D.~Mond.
\newblock Multiple point schemes for corank $1$ maps.
\newblock {\em J. London Math. Soc.}, 2(3):553--567, 1989.

\bibitem{MararNunoANoteOnFiniteDeterminacyForCorank2}
W.~L. Marar and J.~J. Nu{\~n}o-Ballesteros.
\newblock A note on finite determinacy for corank 2 map germs from surfaces to
  3-space.
\newblock {\em Math. Proc. Cambridge Philos. Soc.}, 145(1):153--163, 2008.

\bibitem{Marar2012Double-point-cu}
W.~L. Marar, J.~J. Nu{\~n}o~Ballesteros, and G.~Pe{\~n}afort~Sanchis.
\newblock Double point curves for corank 2 map germs from $\mathbb{C}^2$ to
  $\mathbb{C}^3$.
\newblock {\em Topology Appl.}, 159(2):526--536, 2012.

\bibitem{Mond1985On-the-clasific}
D.~Mond.
\newblock On the clasification of germs of maps from $\mathbb{R}^2$ to
  $\mathbb{R}^3$.
\newblock {\em Proc. London Math. Soc.}, 50(2):333--369, 1985.

\bibitem{MondSomeRemarks}
D.~Mond.
\newblock Some remarks on the geometry and classification of germs of maps from
  surfaces to 3-space.
\newblock {\em Topology}, 26(3):361--383, 1987.

\bibitem{Mond:1991}
D.~Mond.
\newblock {\em Vanishing cycles for analytic maps, Singularity Theory and
  Applications (Warwick 1989)}, volume 1462.
\newblock Springer, New York, 1991.

\bibitem{Mond:1995}
D.~Mond.
\newblock Looking at bent wires -- {$\mathscr A_e$}-codimensions and the
  vanishing topology of parametrised curve singularities,.
\newblock {\em Math. Proc. Cambridge Phil. Soc.}, 117(2):213--222, 1995.

\bibitem{Mond:2016}
D.~Mond.
\newblock Disentanglements of corank 2 map-germs: two examples.
\newblock arXiv:1610.01328, 2016.

\bibitem{Noether:1916}
E.~Noether.
\newblock Der endlichkeitssatz der invarianten endlicher gruppen.
\newblock {\em Math. Ann.}, 77:89--93, 1916.

\bibitem{Nuno-Ballesteros2015On-multiple-poi}
J.J. Nu{\~n}o-Ballesteros and G.~Pe{\~n}afort~Sanchis.
\newblock Multiple point spaces of finite holomorphic maps.
\newblock {\em The Quarterly Journal of Mathematics}, 68(2):369--390, June
  2017.

\bibitem{Ohmoto:2014}
T.~Ohmoto.
\newblock Singularities of maps and characteristic classes.
\newblock arXiv:1309.0661, 2014.

\bibitem{Ohmoto2014Bifurcation-of-}
T.~Ohmoto, T.~Yoshida, and Y.~Kabata.
\newblock Bifurcation of plane-to-plane map-germs of carank 2.
\newblock {\em Quart. J. Math}, 66:369--391, 2015.

\bibitem{Penafort-Sanchis2014THE-GEOMETRY-OF}
G.~Pe{\~n}afort~Sanchis.
\newblock The geometry of double fold maps.
\newblock {\em Journal of Singularities}, 10:250--263, 2014.

\bibitem{Penafort-Sanchis2015Multiple-point-}
G.~Pe{\~n}afort~Sanchis.
\newblock {\em Multiple point spaces of finite holomorphic maps}.
\newblock PhD thesis, Universitat de Val{\`e}ncia, 2015.

\bibitem{Rimany2002Multiple-Point-}
R.~Rim{\'a}nyi.
\newblock Multiple point formulas -- a new point of view.
\newblock {\em Pacific Journal of Mathematics}, 202(2):475--490, 2002.

\bibitem{Ronga1972La-classe-duale}
F.~Ronga.
\newblock La classe duale aux points doubles d'une application.
\newblock {\em Compositio Mathematica}, 27(2):223--232, 1972.

\bibitem{Ruas:1994}
M.A.S. Ruas.
\newblock On the equisingularity of families of corank 1 generic germs.
\newblock {\em Differential Topology, Foliations and Group Action. Contemporary
  Mathematics}, 161:113--121, 1994.

\bibitem{Ruas:2017}
M.A.S. Ruas and O.N. Silva.
\newblock Whitney equisingularity of families of surfaces in $c^3$.
\newblock arXiv:1608.08290, 2017.

\bibitem{Shephard1954Finite-unitary-}
G.~Shephard and J.~A. Todd.
\newblock Finite unitary reflection groups.
\newblock {\em Canadian J. Math}, 6:274--304, 1954.

\bibitem{Springer:1974}
T.A. Springer.
\newblock Regular elements of finite reflection groups.
\newblock {\em Inventiones mathematicae}, pages 159--198, 1974.

\bibitem{Steinberg:1960}
R.~Steinberg.
\newblock Invariants of finite reflection groups.
\newblock {\em Canad. J. Math.}, 12:616--618, 1960.

\bibitem{Steinberg:1968}
R.~Steinberg.
\newblock Differential equations invariant under finite reflection groups.
\newblock {\em Trans. Amer. Math. Soc., Providence RI}, pages 392--400, 1968.

\bibitem{LeConj}
L{\^e}~D{\~u}ng Tr\'ang.
\newblock Problem session in proceeedings of seminar, {P}lans-sur-{B}ex, 1982.
\newblock {\em Monographies de L'Enseignement Mathematique}, 31, 1983.

\bibitem{WallFiniteDeterminacyOfSmoothMapGerms}
C.~T.~C. Wall.
\newblock Finite determinacy of smooth map germs.
\newblock {\em Bull. London Math. Soc.}, 13:481--539, 1981.

\bibitem{Whitney:1936}
H.~Whitney.
\newblock Differentiable manifolds.
\newblock {\em Annals of Mathematics}, 37(3):645--680, 1936.

\bibitem{Whitney:1943}
H.~Whitney.
\newblock The general type of singularity of a set of {$2n-1$} smooth functions
  of $n$ variables.
\newblock {\em Duke Math. J.}, 10(1):161--172, 1943.

\bibitem{Whitney44TheSingularitiesOfSmooth}
H.~Whitney.
\newblock The singularities of a smooth $n$-manifold in {$(2n-1)$}-space.
\newblock {\em Annals of Mathematics}, 45(2):247--293, April 1944.

\bibitem{Wilkinson:1991}
T.~Wilkinson.
\newblock {\em The Geometry of Folding Maps}.
\newblock PhD thesis, University of Newcastle Upon Tyne, 1991.

\end{thebibliography}
\bibliographystyle{plain} 

\end{document}